\documentclass[11pt]{article}
\usepackage[tbtags]{amsmath}
\usepackage{amssymb}
\usepackage{amsthm}
\usepackage[misc]{ifsym}
\usepackage{cases}
\usepackage{mathrsfs}
\usepackage{bbm}
\usepackage{color}
\newcommand{\esssup}{\mathop{\mathrm{esssup}}}

\numberwithin{equation}{section}
\normalsize
\title{\bf A Risk-Sensitive Global Maximum Principle for Controlled Fully Coupled FBSDEs with Applications  \thanks{This work is supported by National Key Research $\&$ Development Program of China (2022YFA1006104), National Natural Science Foundations of China (11971266, 11831010), and Shandong Provincial Natural Science Foundations (ZR2022JQ01, ZR2020ZD24, ZR2019ZD42).}}
\author{\normalsize  Jingtao Lin\thanks{\it School of Mathematics, Shandong University, Jinan 250100, P.R. China, E-mail: linjingtao@mail.sdu.edu.cn} , Jingtao Shi\thanks{Corresponding author. \it School of Mathematics, Shandong University, Jinan 250100, P.R. China, E-mail: shijingtao@sdu.edu.cn}}
\newtheorem{mythm}{Theorem}[section]

\newtheorem{mylem}{Lemma}[section]
\newtheorem{Remark}{Remark}[section]
\newtheorem{myassum}{Assumption}%[section]
\begin{document}
\maketitle

\noindent{\bf Abstract:}\quad This paper is concerned with a kind of risk-sensitive optimal control problem for fully coupled forward-backward stochastic systems. The control variable enters the diffusion term of the state equation and the control domain is not necessarily convex. A new global maximum principle is obtained without assuming that the value function is smooth. The maximum condition, the first- and second-order adjoint equations heavily depend on the risk-sensitive parameter. An optimal control problem with a fully coupled linear forward-backward stochastic system and an exponential-quadratic cost functional is discussed. The optimal feedback control and optimal cost are obtained by using Girsanov's theorem and completion-of-squares approach via risk-sensitive Riccati equations. A local solvability result of coupled risk-sensitive Riccati equations is given by Picard-Lindel$\ddot{o}$f's Theorem.

\vspace{2mm}

\noindent{\bf Keywords:}\quad forward-backward stochastic differential equation, risk-sensitive control, maximum principle, linear-exponential-quadratic control.

\vspace{2mm}

\noindent{\bf Mathematics Subject Classification:}\quad93E20,\quad60H10,\quad 35K15

\section{Introduction}

Originating from the concept of utility, risk sensitivity is widely related to mathematical finance, such as optimal investment (Fleming and Sheu \cite{FS99,FS00}), asset management (Bielecki and Pliska \cite{BP99}, Davis and Lleo \cite{DL11}) and portfolio optimization (Nagai and Peng \cite{NP02}, Nagai \cite{Nagai03}), etc. When taking into account the controller's risk preference, it is natural to consider the risk-sensitive optimal control problems, where usually some risk-sensitive parameters/indices are introduced. Risk-sensitive optimal control problems have close connections with linear-exponential-quadratic Gaussian (LEQG, for short) problems (Jacobson \cite{Jacobson73}, Whittle \cite{Whittle81}, Bensoussan and van Schuppen \cite{BVS85}, Duncan \cite{Duncan13}), robust $H_\infty$ control problems (Ba\c{s}ar \cite{Basar99}, Lim and Zhou \cite{LZ01}) and differential games (James \cite{James92}). The topic of risk-sensitive optimal control problems remains one of most active research fields in optimal control theory for a long time.

Among others, Pontryagin's maximum principle is a useful tool to solve risk-sensitive optimal control problems. Pioneer work for stochastic systems can be seen in Whittle \cite{Whittle90, Whittle91}, Charalambous and Hiley \cite{CH96-1,CH96-2}. Assuming the smoothness of the value function and using the relation between the global stochastic maximum principle of Peng \cite{Peng90} and the dynamic programming principle of Yong and Zhou \cite{YZ99}, Lim and Zhou \cite{LZ05} established a global maximum principle for the risk-sensitive optimal control problem, with the following exponential-of-integral cost functional
\begin{equation}\label{Risk-sensitive cost functional of LimZhou2005}
	J^\theta(u(\cdot))=\mathbb{E}\Big[e^{\theta\big[\int_0^Tf(t,x(t),u(t))\,dt+\phi(x(T))\big]}\Big],
\end{equation}
where $\theta\in\mathbb{R}$ is the risk-sensitive index. The state process $x(\cdot)$ follows a controlled stochastic differential equation (SDE, for short):
\begin{equation}\label{state of LimZhou2005}
\left\{
	\begin{aligned}
		dx(t)&=b(t,x(t),u(t))\,dt+\sigma(t,x(t),u(t))\,dW(t),\quad t\in[0,T],\\
		x(0)&=x_0,
	\end{aligned}
	\right.
\end{equation}
and the control domain $U$ is not necessarily convex. The result in \cite{LZ05} is applied to an optimal portfolio choice problem by Shi and Wu \cite{SW12}, and extended to the jump-diffusion case by Shi and Wu \cite{SW11}. By introducing a non-linear transformation from El Karoui and Hamad\`{e}ne \cite{EH03}, Djehiche et al. \cite{DTT15} eliminated the smoothness assumption of the value function, and obtained a risk-sensitive global maximum principle for optimal control of SDEs of mean-field type, where the expectations of the state and control variables are included in the drift and diffusion terms. The result in \cite{LZ05} is also extended to the optimal control of backward SDEs (BSDEs, for short) by Chala \cite{Chala17-1}, where a risk-sensitive local maximum principle is proved. (The word ``local" here means that the control domain is assumed to be convex.) In \cite{Chala17-1}, the controlled BSDE is of the following Pardoux-Peng's form
\begin{equation}\label{BSDE of Chala2017}
\left\{
\begin{aligned}
-dy(t)&=-g(t,y(t),z(t),u(t))\,dt-z(t)\,dW(t),\quad t\in[0,T],\\
  y(T)&=\xi,
\end{aligned}	
\right.
\end{equation}
and the exponential-of-integral cost functional is the following
\begin{equation}\label{Risk-sensitive cost functional of Chala2017}
	J^\theta(u(\cdot))=\mathbb{E}\Big[e^{\theta\big[\int_0^Tf(t,y(t),z(t),u(t))\,dt+\psi(y(0))\big]}\Big].
\end{equation}
Some sufficient conditions are supplemented in Chala \cite{Chala17-2}.

Recently, motivated by some cash flow modeling and controlling problems, Khallout and Chala \cite{KC20} considered a risk-sensitive optimal control problem, where the state equation is the following controlled fully coupled forward-backward SDE (FBSDE, for short):
\begin{equation}\label{FBSDE of Khallout and Chala 2020}
	\left\{
	\begin{aligned}
		dx(t)&=b(t,x(t),y(t),z(t),u(t))\,dt+\sigma(t,x(t),y(t),z(t),u(t))\,dW(t),\\
		-dy(t)&=g(t,x(t),y(t),z(t),u(t))\,dt-z(t)\,dW(t),\quad t\in[0,T],\\
		x(0)&=x_0,\quad y(T)=\xi,
	\end{aligned}
	\right.
\end{equation}
with the exponential-of-integral cost functional
\begin{equation}\label{Risk-sensitive cost functional of Khallout and Chala 2020}
	J^\theta(u(\cdot))=\mathbb{E}\Big[e^{\theta\big[\int_0^Tf(t,x(t),y(t),z(t),u(t))\,dt+\phi(x(T))+\psi(y(0))\big]}\Big].
\end{equation}
By assuming that the control domain is convex, necessary and sufficient maximum principles are obtained. Results for mean-field type control is soon given in Chala and Hafayed \cite{CH20}. Both these results generalized the pioneer work of Wu \cite{Wu98} from risk-neutral case to risk-sensitive case. Risk-sensitive zero-sum stochastic differential games are studied by Moon et al. \cite{MDB19}.

When the control domain is not necessarily convex, Moon \cite{Moon20} extended previous works of Yong \cite{Yong10} and Wu \cite{Wu13} from risk-neutral case to the risk-sensitive case. However, unknown parameters are unavoidable when applying Ekeland's variational principle to treat some embedded problems with state constraints, which means that the results in \cite{Moon20} are not ideal.

Recently, Hu et al. \cite{HJX18} obtained a novel global maximum principle for optimal control of fully coupled forward-backward stochastic systems, where the state equation is the same as (\ref{FBSDE of Khallout and Chala 2020}), with the recursive cost functional, that is,
\begin{equation}\label{cost functional of HJX2018}
	J(u(\cdot))=y(0).
\end{equation}
The main techniques in \cite{HJX18} is a new first- and second-order spike variational method which was first introduced by Hu \cite{Hu17} and a new decoupling relation between first- and second-order variations of $x(\cdot)$, $y(\cdot)$ and $z(\cdot)$. For ease of use, very recently, Lin and Shi \cite{LS-CCDC2023submission} extended the above result from the recursive cost functional (\ref{cost functional of HJX2018}) to the case with the cost functional of the following general form:
\begin{equation}\label{cost functional of Lin and Shi CCDC2023submission}
	J(u(\cdot))=\mathbb{E}\bigg[\int_0^Tl(t,x(t),y(t),z(t),u(t))\,dt+f(x(T),y(0))\bigg].
\end{equation}

Motivated by the above literatures, in this paper, we consider the global maximum principle for the risk-sensitive optimal control problem, with the controlled fully coupled FBSDE:
\begin{equation}\label{state of this paper}
	\left\{
	\begin{aligned}
		dX(t)&=b(t,X(t),Y(t),Z(t),u(t))\,dt+\sigma(t,X(t),Y(t),Z(t),u(t))\,dW(t),\\
		-dY(t)&=g(t,X(t),Y(t),Z(t),u(t))\,dt-Z(t)\,dW(t),\quad t\in[0,T],\\
		X(0)&=x_0,\quad Y(T)=\phi(X(T)),
	\end{aligned}
	\right.
\end{equation}
and the general exponential-of-integral cost functional:
\begin{equation}\label{cost functional of this paper}
	J^\theta(u(\cdot))=\mathbb{E}\Big[e^{\theta\big[\int_0^Tl(t,X(t),Y(t),Z(t),u(t))\,dt+f(X(T),Y(0))\big]}\Big].
\end{equation}
To the best of the authors' knowledge, the risk-sensitive global maximum principle for fully coupled FBSDEs obtained in this paper has not been studied in the existing literature. In Section 3 of this paper, new first- and second-order adjoint equations are obtained. Especially, the new first-order adjoint equation (\ref{First-order Risk-sensitive Adjoint Equation(mnk)}) is a non-coupled FBSDE.

For linear-quadratic (LQ, for short) optimal control problems in risk-neutral case, there exists a long list of literatures. However, for LQ optimal control problems of fully coupled forward-backward stochastic systems, few papers exist, as far as we know, except Wang et al. \cite{WGW03}, Shi and Wu \cite{SW09,SW10}, Huang and Shi \cite{HS12}, and Hu et al. \cite{HJX18} for some special examples. Very recently, motivated by optimization under rational expectations, Hu et al. \cite{HJX22} gave a systematic research for the LQ optimal control problem of fully coupled forward-backward stochastic systems, and a new decoupling technique is proposed to obtain the feedback optimal control.

Linear-exponential-quadratic (LEQ, for short) problem, introduced by Jacobson\cite{Jacobson73}, is the risk-sensitive counterpart of LQ optimal control problems. The reader is referred to Djehiche et al. \cite{DTT15}, Duncan \cite{Duncan13}, Lim and Zhou \cite{LZ05}, Moon et al. \cite{MDB19} , Shi and Wu \cite{SW11}, Ma and Wang \cite{MW-22} and the references therein for the theory of LEQ problem and its wide ranges of applications. In Section 5 of this paper, distinguished from the above literature reviewed, we consider an LEQ problem of fully coupled forward-backward stochastic systems. Especially, all square terms and cross terms of $X$, $Y$, $Z$ and $u$ can appear in the cost functional. Inspired by Djehiche et al. \cite{DTT15} and Duncan \cite{Duncan13}, the optimal feedback control and optimal cost are obtained by using Girsanov's Theorem and completion-of-squares approach via risk-sensitive Riccati equations.

Compared with the existing results, the innovation of this paper can be summarized along the following lines:
\begin{itemize}
	\item The control system is a fully coupled FBSDE, the cost functional is a general exponential-of-integral form, the control domain is not necessarily convex, and the assumption of smoothness of value function is not needed.
	\item By introducing a non-linear transformation, new first- and second-order adjoint equations, which depend on the risk-sensitive parameter, are obtained. Especially, the first-order adjoint equations are decoupled forward-backward SDEs, which are more useful in practice.
	\item The global maximum principle does not contain unknown parameters, which suggests that our result is more convenient to apply.
    \item An LEQ problem of fully coupled linear forward-backward stochastic systems is discussed. All square terms and cross terms of $X$, $Y$, $Z$ and $u$ can appear in the cost functional. The optimal feedback control and optimal cost are obtained by using Girsanov's Theorem and completion-of-squares approach via risk-sensitive Riccati equations. A local solvability result of coupled risk-sensitive Riccati equations (\ref{LEQ-alpha1}) and (\ref{LEQ-beta1}) is given by Picard-Lindel$\ddot{o}$f's Theorem.
\end{itemize}

The rest of this paper is organized as follows. In Section 2, we state the problem and assumptions. The main result is established in Section 3 and its detailed proof is given in Section 4. In Section 5, an LEQ problem
of fully coupled forward-backward stochastic systems is discussed. Finally, some concluding remarks are given in Section 6.

{\bf Notations.} In this paper, the $n$-dimensional Euclidian space is denoted by $\mathbb{R}^n$, and $x^\top$ is the transpose of $x\in\mathbb{R}^n$. The inner product of $x,y\in\mathbb{R}^n$ is denoted by $\langle x,y\rangle$, and the norm of $x\in \mathbb{R}^n$ is denoted by $|x|=\langle x,x\rangle^{\frac{1}{2}}$. Let $\mathcal{S}^n$, $\mathcal{S}^n_+$, $\bar{\mathcal{S}}^n_+$ be the spaces of $n\times n$ dimensional symmetric, symmetric positive definite, symmetric positive semi-definite matrices, respectively. $D\psi$ is the gradient of the function $\psi(x,y,z)$ with respect to $x,y,z$, and $D^2\psi$ is the Hessian matrix of it with respect to $x,y,z$.

\section{Problem formulation}

Let $T\in(0,+\infty)$ be fixed and $(\Omega,\mathcal{F},\mathbb{P},\left\{\mathcal{F}_t\right\}_{t\geq0})$ be a complete filtered probability space on which a one-dimensional standard Brownian motion $W(\cdot)$ is defined such that $\mathbb{F}:=\left\{\mathcal{F}_t\right\}_{t\geq0}$ is the natural filtration generated by $W(\cdot)$, augmented by all $\mathbb{P}$-null sets in $\mathcal{F}$. We introduced some notations of spaces used throughout this paper.

$L^p_{\mathcal{F}_T}(\Omega;\mathbb{R}^n)$: the space of $\mathcal{F}_T$-measurable $\mathbb{R}^n$-valued random variables $\eta$ such that
$$
||\eta||_p:=\big(\mathbb{E}[|\eta|^p]\big)^{\frac{1}{p}}<\infty.
$$

$L^p_{\mathcal{F}}([0,T];\mathbb{R}^n)$: the space of $\mathbb{F}$-adapted and $p$th integrable stochastic processes $f(\cdot)$ on $[0,T]$ such that
$$
\mathbb{E}\bigg[\int_0^T|f(t)|^p\,dt\bigg]<\infty.
$$

$L^\infty_{\mathcal{F}}([0,T];\mathbb{R}^n)$: the space of $\mathbb{F}$-adapted and uniformly bounded stochastic processes on $[0,T]$ such that
$$
||f||_\infty:=\esssup\limits_{(t,\omega)\,\in[0,T]\times\Omega}|f(t)|<\infty.
$$

$L^p_{\mathcal{F}}(\Omega;C([0,T];\mathbb{R}^n))$: the space of $\mathbb{F}$-adapted continuous stochastic processes $f(\cdot)$ on $[0,T]$ such that
$$
\mathbb{E}\bigg[\sup\limits_{0\leq t\leq T}|f(t)|^p\bigg]<\infty.
$$

In this paper, we study a risk-sensitive optimal control problem for a fully coupled forward-backward stochastic system. The state equation is:
\begin{equation}\label{stateEquation}
	\left\{
	\begin{aligned}
		dX(t)&=b(t,X(t),Y(t),Z(t),u(t))\,dt+\sigma(t,X(t),Y(t),Z(t),u(t))\,dW(t),\\
		-dY(t)&=g(t,X(t),Y(t),Z(t),u(t))\,dt-Z(t)\,dW(t),\quad t\in[0,T],\\
		X(0)&=x_0,\quad Y(T)=\phi(X(T)),
	\end{aligned}
	\right.
\end{equation}
where $(X(\cdot),Y(\cdot),Z(\cdot))$ is the state process triple valued in $\mathbb{R}^n\times\mathbb{R}\times\mathbb{R}$.
The coefficients $b,\sigma:[0,T]\times\mathbb{R}^n\times\mathbb{R}\times\mathbb{R}\times U\to\mathbb{R}^n$, $g:[0,T]\times\mathbb{R}^n\times\mathbb{R}\times\mathbb{R}\times U\to\mathbb{R}$ and $\phi:\mathbb{R}^n\to\mathbb{R}$ are given functions. For $b=b,\sigma$, set $b(\cdot)=(b^{(1)}(\cdot),b^{(2)}(\cdot),\cdots,b^{(n)}(\cdot))^\top\in\mathbb{R}^n$.

An admissible control $u(\cdot)$ is an $\mathbb{F}$-adapted process valued in the nonempty control domain $U\subseteq\mathbb{R}^k$ such that
$$
\sup\limits_{0\leq t\leq T}\mathbb{E}\big[|u(t)|^8\big]<\infty.
$$
Denote the admissible control set by $\mathcal{U}[0,T]$.

In risk-sensitive optimal control problems, it is taken into account the risk attitude (risk-seeking, risk-averse or risk-neutral) of the controller in choosing an appropriate criterion to minimize. In this paper, we consider the cost functional of the exponential-of integral type:
\begin{equation}\label{R-S CostFunctional}
	J^{RS}(u(\cdot)):=\mathbb{E}\Big[e^{\theta\big[\int_0^Tl(t,X(t),Y(t),Z(t),u(t))\,dt+f(X(T),Y(0))\big]}\Big],
\end{equation}
where $l:[0,T]\times\mathbb{R}^n\times\mathbb{R}\times\mathbb{R}\times U\to\mathbb{R}^n$, $f:\mathbb{R}^n\times\mathbb{R}\to\mathbb{R}$ are given functions, and $\theta$ is the risk-sensitive index.

Next we explain the risk-sensitive index $\theta$ in detail. Define
$$L_T(u(\cdot)):=\int_0^Tl(t,X(t),Y(t),Z(t),u(t))\,dt+f(X(T),Y(0)),$$
then the certainty-equivalent expectation is given by
\begin{equation*}
	\begin{aligned}
		\Theta_{\theta}(u(\cdot))&:=\frac{1}{\theta}\ln J^{RS}(u(\cdot))=\frac{1}{\theta}\ln\mathbb{E}\big[e^{\theta L_T(u(\cdot))}\big].
	\end{aligned}
\end{equation*}
For a sufficiently small $\theta$, $\Theta_{\theta}(u(\cdot))$ has the following Taylor's expansion:
\begin{equation*}
	\Theta_{\theta}(u(\cdot))=\mathbb{E}\big[L_T(u(\cdot))\big]+\frac{\theta}{2}\mbox{var}\big[L_T(u(\cdot))\big]+o(\theta).
\end{equation*}
If $\theta<0$, the term $\mbox{var}\big[\Theta_T(u(\cdot))\big]$, as a measure of risk, improves the performance $\Theta_{\theta}(u(\cdot))$. Hence, in this case, the controller is called risk-seeking. Similarly, $\theta>0$ corresponds to the risk-averse controller.

In this paper, the problem is to find an optimal control $\bar{u}(\cdot)$ over $\mathcal{U}[0,T]$ to optimize the cost functional (\ref{R-S CostFunctional}), i.e.,
\begin{equation}\label{RS problem}
J^{RS}(\bar{u}(\cdot))=\inf\limits_{u(\cdot)\in\,\mathcal{U}[0,T]}J^{RS}(u(\cdot)).
\end{equation}
The corresponding $(\bar{X}(\cdot),\bar{Y}(\cdot),\bar{Z}(\cdot))$ is called the optimal state process triple, and $(\bar{X}(\cdot),\bar{Y}(\cdot)$, $\bar{Z}(\cdot),\bar{u}(\cdot))$ is the optimal quadruple.

We need the following hypotheses.

\begin{myassum}\label{SeparableSubset}
	\begin{itemize}
	The control domain $U$ is a nonempty separable subset of $\mathbb{R}^k$.
	\end{itemize}
\end{myassum}

\begin{myassum}\label{AssumSolvableOfStateEquation}
	\begin{itemize}
		\item[(i)] For $b=b$, $\sigma$, $g$, $\phi$, all $b$, $b_x$, $b_y$, $b_z$ are continuous in $(x,y,z,u)$; $b_x$, $b_y$, $b_z$ are bounded; there exists some constant $L>0$ such that for any $t\in[0,T]$, $x\in\mathbb{R}^n,$ $y\in\mathbb{R}$, $z\in\mathbb{R}$ and $u,u'\in U$, it holds that
		\begin{equation*}
			\begin{array}{cr}
				|b(t,x,y,z,u)|\le L(1+|x|+|y|+|z|+|u|),\\
				|\sigma(t,0,0,z,u)-\sigma(t,0,0,z,u')|\leq L(1+|u|+|u'|).
			\end{array}
		\end{equation*}
        \item[(ii)] $f$ is continuous in $(x,y)$, and there exists some constant $L>0$ such that for any $x\in\mathbb{R}^n,$ $y\in\mathbb{R}$, $|f(t,x,y)|\le L(1+|x|^2+|y|^2)$; $l$ is continuous in $(x,y,z,u)$, and there exists some constant $L>0$ such that for any $x\in\mathbb{R}^n,$ $y\in\mathbb{R}$, $z\in\mathbb{R}$ and $u\in U$, $|l(t,x,y,z,u)|\le L(1+|x|^2+|y|^2+|z|^2+|u|^2)$.
		\item[(iii)] For any $2\leq\beta\leq8$, $\Lambda_{\beta}:=C_{\beta}2^{\beta+1}(1+T^{\beta})c_1^{\beta}<1$, where $c_1=max\{L_2,L_3\}$, $L_2=max\{||b_y||_{\infty},||b_z||_{\infty},||\sigma_y||_{\infty}\}$, $L_3=||\sigma_z||_{\infty}$, $C_{\beta}$ is defined in Lemma A.1 in the appendix of \cite{HJX18} for $L_1=max\{||b_x||_{\infty},||\sigma_x||_{\infty},||g_x||_{\infty},||g_y||_{\infty},||g_z||_{\infty},||\phi_x||_{\infty}\}$.
		\item[(iv)] For $b=b$, $\sigma$, $g$, $\phi$, all $b_{xx}$, $b_{xy}$, $b_{yy}$, $b_{xz}$, $b_{yz}$, $b_{zz}$ are continuous in $(x,y,z,u)$; $b_{xx}$, $b_{xy}$, $b_{yy}$, $b_{xz}$, $b_{yz}$, $b_{zz}$ are bounded.
	\end{itemize}
\end{myassum}

For any given control $u(\cdot)\in\mathcal{U}[0,T]$, Assumption \ref{AssumSolvableOfStateEquation} guarantees that the state equation (\ref{stateEquation}) and all the first- and second-order adjoint equations in Section 3 are uniquely solvable. In fact, some of Assumption \ref{AssumSolvableOfStateEquation} can be relaxed, but we do not give details since the topic of the well-posedness of FBSDEs is not the main target of this paper.

\section{Main result}

In this section, the main result, a Pontryagin's type necessary condition of the risk-sensitive optimal control problem (\ref{RS problem}), is established. The main idea is to transfer (\ref{RS problem}) into an equivalent risk-neutral one, by introducing an auxiliary process. First, we recall a risk-neutral global MP of fully coupled FBSDEs in our accompanying paper \cite{LS-CCDC2023submission}.

\subsection{A risk-neutral global maximum principle}\label{section-SMP-RN}

In this section, we consider the state equation (\ref{stateEquation}) with the cost functional of the following Mayer's type:
\begin{equation}\label{CostFunctionalMayerType}
	J^{RN}(u(\cdot))=\mathbb{E}\big\{\varphi(X(T),Y(0))\big\},
\end{equation}
where $\varphi:\mathbb{R}^n\times\mathbb{R}\to\mathbb{R}$ is a given function satisfies the same assumptions as $f$ in Assumption \ref{AssumSolvableOfStateEquation}.

Let $(\bar{X}(\cdot),\bar{Y}(\cdot),\bar{Z}(\cdot),\bar{u}(\cdot))$ be an optimal quadruple. Let $E_{\epsilon}\subset[0,T]$ is a measurable set with $|E_{\epsilon}|=\epsilon\in(0,T)$. For any $u(\cdot)\in\mathcal{U}[0,T]$, the spike variation of $\bar{u}(\cdot)$ is defined as
\begin{equation}\label{SpikeVariation}
	u^{\epsilon}(t):=\bar{u}(t)I_{E^c_{\epsilon}}(t)+u(t)I_{E_{\epsilon}}(t).
\end{equation}
Let $(X^{\epsilon}(\cdot),Y^{\epsilon}(\cdot),Z^{\epsilon}(\cdot))$ be the corresponding state process triple. For $b=b$, $\sigma$, $g$, $\phi$, $\varphi$ and $x=x$, $y$, $z$, we introduce the following notations:
\begin{equation*}
	\begin{aligned}
		b(t)&:=\psi(t,\bar{X}(t),\bar{Y}(t),\bar{Z}(t),\bar{u}(t)),\quad b_x(t):=b_x(t,\bar{X}(t),\bar{Y}(t),\bar{Z}(t),\bar{u}(t)),\\
		Db(t)&:=Db(t,\bar{X}(t),\bar{Y}(t),\bar{Z}(t),\bar{u}(t)),\quad D^2b(t):=D^2b(t,\bar{X}(t),\bar{Y}(t),\bar{Z}(t),\bar{u}(t)),\\
		\delta b(t)&:=b(t,\bar{X}(t),\bar{Y}(t),\bar{Z}(t),u(t))-b(t),\quad \delta b_x(t):=b_x(t,\bar{X}(t),\bar{Y}(t),\bar{Z}(t),u(t))-b_x(t),\\
		\delta b(t,\Delta)&:=\psi(t,\bar{X}(t),\bar{Y}(t),\bar{Z}(t)+\Delta(t),u(t))-b(t),\\
		\delta b_x(t,\Delta)&:=b_x(t,\bar{X}(t),\bar{Y}(t),\bar{Z}(t)+\Delta(t),u(t))-b_x(t).
	\end{aligned}
\end{equation*}
In the above, $\Delta(\cdot)$ is an $\mathbb{F}$-adapted process which is the solution of the following algebra equation:
\begin{equation}\label{algebraEquation}
	\Delta(t)=p^\top(t)\delta\sigma(t,\Delta(t)),
\end{equation}
where $p(\cdot)$ is the solution of the following (\ref{first-order-adjoint(pq)}). The detail can be found in \cite{HJX18}.

We first introduce the first- and second-order decoupling fields, to describe the relation among $\bar{X}(\cdot)$, $\bar{Y}(\cdot)$ and $\bar{Z}(\cdot)$. Interested readers may refer to \cite{HJX18} and \cite{LS-CCDC2023submission} for more details. In this paper, for $\Phi(\cdot)\in\mathbb{R}^d$, $\Psi(\cdot)\in\mathbb{R}$, $x\in\mathbb{R}^n$ and $y\in\mathbb{R}$, we have the following notions:
\begin{equation*}
	\Phi_x=\begin{bmatrix}
		\Phi^1_{x_1}&  \cdots& \Phi^1_{x_n}  \\
		\vdots& \ddots &  \vdots \\
		\Phi^d_{x_1}& \cdots & \Phi^d_{x_n}  \\
	\end{bmatrix},\
\Phi_y=\begin{bmatrix}
	\Phi^1_y\\
	\vdots\\
	\Phi^d_y \\
\end{bmatrix},\
\Psi_x=\begin{bmatrix}
	\Psi_{x_1} \\
	\vdots\\
	\Psi_{x_n}\\
\end{bmatrix},\
\Psi_{xx}=\begin{bmatrix}
	\Psi_{x_1x_1}&  \cdots& \Psi_{x_1x_n}  \\
	\vdots& \ddots &  \vdots \\
	\Psi_{x_nx_1}& \cdots & \Psi_{x_nx_n}  \\
\end{bmatrix}.
\end{equation*}

We introduce the following first-order decoupling field:
\begin{equation}\label{first-order-adjoint(pq)}
	\left\{
	\begin{aligned}
		-dp(t)&=\Big\{ g_x(t)+p(t)g_y(t)+K_1(t)g_z(t)+b^\top_x(t)p(t)+p(t)b^\top_y(t)p(t)\\
             &\qquad +K_1(t)b^\top_z(t)p(t)+\sigma^\top_x(t)q(t)+p(t)\sigma^\top_y(t)q(t)+K_1(t)\sigma^\top_z(t)q(t)\Big\}\,dt\\
             &\quad -q(t)\,dW(t),\quad t\in[0,T],\\
		 p(T)&=\phi_x(\bar{X}(T)),
	\end{aligned}
	\right.
\end{equation}
which admits a unique solution $(p(\cdot),q(\cdot))\in L^2_{\mathcal{F}}(\Omega;C([0,T];\mathbb{R}^n))\times L^2_{\mathcal{F}}([0,T];\mathbb{R}^n)$,
where
\begin{equation}\label{K_1}
	\mathbb{R}^n\ni K_1(t):=(1-p^\top(t)\sigma_z(t))^{-1}\big[\sigma^\top_x(t)p(t)+p(t)\sigma^\top_y(t)p(t)+q(t)\big],
\end{equation}
and the following second-order decoupling field:	
\begin{equation}\label{second-order-adjoint(PQ)}
	\left\{
	\begin{aligned}
		-dP(t)&=\Big\{G_y(t)P(t)+G_z(t)K_2(t)+\big[b^\top_x(t)+p(t)b^\top_y(t)+K_1(t)b^\top_z(t)\big]P(t)\\
		      &\quad+P(t)\big[b_x(t)+b_y(t)p^\top(t)+b_z(t)K^\top_1(t)\big]\\
		      &\quad+\big[\sigma^\top_x(t)+p(t)\sigma^\top_y(t)+K_1(t)\sigma^\top_z(t)\big]Q(t)\\
              &\quad+Q(t)\big[\sigma_x(t)+\sigma_y(t)p^\top(t)+\sigma_z(t)K^\top_1(t)\big]\\
		      &\quad+\big[\sigma^\top_x(t)+p(t)\sigma^\top_y(t)+K_1(t)\sigma^\top_z(t)\big]P(t)\big[\sigma_x(t)+\sigma_y(t)p^\top(t)+\sigma_z(t)K^\top_1(t)\big]\\
		      &\quad+[I_n,p(t),K_1(t)]D^2G(t)[I_n,p(t),K_1(t)]^\top\Big\}\,dt-Q(t)\,dW(t),\quad t\in[0,T],\\
		  P(T)&=\phi_{xx}(\bar{X}(T)),
	\end{aligned}
	\right.
\end{equation}
which admits a unique solution $(P(\cdot),Q(\cdot))\in L^2_{\mathcal{F}}(\Omega;C([0,T];\mathcal{S}^n))\times L^2_{\mathcal{F}}([0,T];\mathcal{S}^n)$, and
\begin{equation}\label{K_2}
	\begin{aligned}
		\mathbb{R}^{n\times n}\ni K_2(t)&:=(1-p^\top(t)\sigma_z(t))^{-1}\Big\{p^\top(t)\sigma_y(t)P(t)+Q(t)+\big[\sigma^\top_x(t)+p(t)\sigma^\top_y(t)\\
                                        &\qquad +K_1(t)\sigma^\top_z(t)\big]P(t)+P(t)\big[\sigma_x(t)+\sigma_y(t)p^\top(t)+\sigma_z(t)K^\top_1(t)\big]\\
                                        &\qquad +p^\top(t)[I_n,p,K_1]D^2\sigma(t)[I_n,p,K_1]^\top\Big\},
	\end{aligned}
\end{equation}
where
\begin{equation}\label{G()}
	G(t,x,y,z,u,p,q):=g(t,x,y,z,u)+p^\top b(t,x,y,z,u)+q^\top\sigma(t,x,y,z,u).
\end{equation}

Moreover, in this paper, we have the following notions:
$$p^\top(t)[I_n,p,K_1]D^2b(t)[I_n,p,K_1]^\top:=\sum_{i=1}^{n}p^{(i)}(t)[I_n,p,K_1]D^2b^{(i)}(t)[I_n,p,K_1]^\top,$$
and $[I_n,p,K_1]D^2b^{(i)}[I_n,p,K_1]^\top$, $i=1,2,\cdots,n,$ denotes the $n\times n$ matrix:
\begin{equation}
	\begin{aligned}
		[I_n,p,K_1]D^2b^{(i)}[I_n,p,K_1]^\top&:= I_nb^{(i)}_{xx}I_n+I_nb^{(i)}_{xy}p^{\top}+I_nb^{(i)}_{xz}K_1^{\top}+pb^{(i)}_{yx}I_n+pb^{(i)}_{yy}p^{\top}\\
		&\quad\ +pb^{(i)}_{yz}K_1^{\top}+K_1b^{(i)}_{zx}I_n+K_1b^{(i)}_{zy}p^{\top}+K_1b^{(i)}_{zz}K_1^{\top},
	\end{aligned}
\end{equation}
similar as $[I_n,p,K_1]D^2\sigma^{(i)}(t)[I_n,p,K_1]^\top$ and $[I_n,p,K_1]D^2g(t)[I_n,p,K_1]^\top$.

To establish estimations of variations, the following technical assumption is needed (see \cite{HJX18}).
\begin{myassum}\label{q is bounded}
	The solution $q(\cdot)$ of (\ref{first-order-adjoint(pq)}) is a bounded process.
\end{myassum}

Next, we recall $(\lambda(\cdot),\xi(\cdot),\eta(\cdot))\in L^2_{\mathcal{F}}(\Omega;C([0,T];\mathbb{R}^n))\times L^2_{\mathcal{F}}(\Omega;C([0,T];\mathbb{R}^n))\times L^2_{\mathcal{F}}([0,T];\mathbb{R}^n)$ and $(\Gamma(\cdot),\Lambda(\cdot))\in L^2_{\mathcal{F}}(\Omega;C([0,T];\mathcal{S}^n))\times L^2_{\mathcal{F}}([0,T];\mathcal{S}^n)$, which are the solutions of the following first- and second-order adjoint equations, respectively:
\begin{equation}\label{First-Order-Adjoint(mnk)}
	\left\{
	\begin{aligned}
	  d\lambda(t)&=\Big\{\big[\xi^\top(t)b_z(t)+\eta^\top(t)\sigma_z(t)\big](1-p^\top(t)\sigma_z(t))^{-1}p^\top(t)\sigma_y(t)+\xi^\top(t)b_y(t)\\
		         &\qquad +\eta^\top(t)\sigma_y(t)+\lambda(t)\big[G_y(t)+(1-p^\top(t)\sigma_z(t))^{-1}g_z(t)p^\top(t)\sigma_y(t)\big]\Big\}\,dt\\
		         &\quad +\Big\{\big[\xi^\top(t)b_z(t)+\eta^\top(t)\sigma_z(t)\big](1-p^\top(t)\sigma_z(t))^{-1}p^\top(t)\sigma_z(t)+\xi^\top(t)b_z(t)\\
		         &\qquad +\eta^\top(t)\sigma_z(t)+\lambda(t)\big[G_z(t)+(1-p^\top(t)\sigma_z(t))^{-1}g_z(t)p^\top(t)\sigma_z(t)\big]\Big\}\,dW(t),\\
		 -d\xi(t)&=\Big\{\big[b^\top_x(t)+p(t)b^\top_y(t)+K_1(t)b^\top_z(t)\big]\xi(t)+\big[\sigma^\top_x(t)+p(t)\sigma^\top_y(t)\\
		         &\qquad +K_1(t)\sigma^\top_z(t)\big]\eta(t)\Big\}\,dt-\eta(t)\,dW(t),\quad t\in[0,T],\\
	   \lambda(0)&=\varphi_y(\bar{X}(T),\bar{Y}(0)),\quad \xi(T)=\varphi_x(\bar{X}(T),\bar{Y}(0)),
	\end{aligned}
	\right.
\end{equation}
\begin{equation}\label{Second-Order_Adjoint(MN)}
	\left\{
	\begin{aligned}
 -d\Gamma(t)&=\Big\{\Gamma(t)\big[b_x(t)+b_y(t)p^\top(t)+b_z(t)K^\top_1(t)\big]+\big[b^\top_x(t)+p(t)b^\top_y(t)\\
		    &\qquad +K_1(t)b^\top_z(t)\big]\Gamma(t)+\big[\sigma^\top_x(t)+p(t)\sigma^\top_y(t)+K_1(t)\sigma^\top_z(t)\big]\Gamma(t)\big[\sigma_x(t)\\
		    &\qquad +\sigma_y(t)p^\top(t)+\sigma_z(t)K^\top_1(t)\big]+\big[\sigma^\top_x(t)+p(t)\sigma^\top_y(t)+K_1(t)\sigma^\top_z(t)\big]\Lambda(t)\\
		    &\qquad +\Lambda(t)\big[\sigma_x(t)+\sigma_y(t)p^\top(t)+\sigma_z(t)K_1^\top(t)\big]+\xi^\top(t)\big[b_y(t)P(t)\\
		    &\qquad +b_z(t)K_2(t)+[I_n,p(t),K_1(t)]D^2b(t)[I_n,p(t),K_1(t)]^\top\big]+\eta^\top(t)\big[\sigma_y(t)P(t)\\
		    &\qquad +\sigma_z(t)K_2(t)+[I_n,p(t),K_1(t)]D^2\sigma(t)[I_n,p(t),K_1(t)]^\top\big]\Big\}\,dt\\
		    &\quad -\Lambda(t)\,dW(t),\quad t\in[0,T],\\
   \Gamma(T)&=\varphi_{xx}(\bar{X}(T),\bar{Y}(0)).
	\end{aligned}
	\right.
\end{equation}

\begin{Remark}\label{RemarkOnSolvabilityOf(p,q)and(P,Q)}
\cite{HJX18} first established (\ref{first-order-adjoint(pq)}) and (\ref{second-order-adjoint(PQ)}) in one-dimensional form to decouple the first- and second-order variational equations. The above high-dimension form is given in \cite{LS-CCDC2023submission}.
 Under some mild conditions, the unique solvability of (\ref{algebraEquation}) and (\ref{first-order-adjoint(pq)}) had been proved in Theorem 3.6 and Lemma 3.9 of \cite{HJX18}.
(\ref{First-Order-Adjoint(mnk)}) is a non-coupled FBSDE and Assumption \ref{AssumSolvableOfStateEquation} guarantees its unique solvability.
 (\ref{second-order-adjoint(PQ)}) is a linear BSDE with uniformly Lipschitz coefficients, which has a unique solution.	
\end{Remark}

We define the risk-neutral Hamiltonian function $\mathcal{H}^{RN}:[0,T]\times\mathbb{R}^n\times\mathbb{R}\times\mathbb{R}\times U\times\mathbb{R}^n\times\mathbb{R}^n\times\mathcal{S}^n\times\mathbb{R}^n\times\mathbb{R}^n\times\mathbb{R}^n\times\mathcal{S}^n\to\mathbb{R}$:
\begin{equation}\label{Hamiltonian-RN}
	\begin{aligned}
		\mathcal{H}&^{RN}(t,x,y,z,u,p,q,P,\lambda,\xi,\eta,\Gamma)\\
		&:=\big(\xi^\top+\lambda p^\top\big)b(t,x,y,z+\Delta(t),u)+\big(\eta^\top+\lambda q^\top\big)\sigma(t,x,y,z+\Delta(t),u)\\
		&\quad +\lambda g(t,x,y,z+\Delta(t),u)+\frac{1}{2}\big[\sigma(t,x,y,z+\Delta(t),u)-\sigma(t,\bar{X}(t),\bar{Y}(t),\bar{Z}(t),\bar{u}(t))\big]^\top\\
		&\quad \times(\Gamma+\lambda P)\big[\sigma(t,x,y,z+\Delta(t),u)-\sigma(t,\bar{X}(t),\bar{Y}(t),\bar{Z}(t),\bar{u}(t))\big],
	\end{aligned}
\end{equation}
where $\Delta(\cdot)$ is the solution of (\ref{algebraEquation}) corresponding to $u(t)=u$.

The following lemma is from \cite{LS-CCDC2023submission}.
\begin{mylem}\label{SMP-MayerType}
	Suppose Assumption \ref{SeparableSubset}, Assumption \ref{AssumSolvableOfStateEquation} and Assumption \ref{q is bounded} hold. Let (\ref{first-order-adjoint(pq)}) and (\ref{algebraEquation}) have unique solutions $(p(\cdot),q(\cdot))\in L^2_{\mathcal{F}}(\Omega;C([0,T],\mathbb{R}))\times L^2_{\mathcal{F}}([0,T];\mathbb{R})$ and $\Delta(\cdot)\in L^2_{\mathcal{F}}([0,T];\mathbb{R}^n)$ respectively. Let $\bar{u}(\cdot)\in\mathcal{U}[0,T]$ be optimal and $(\bar{X}(\cdot),\bar{Y}(\cdot),\bar{Z}(\cdot))$ be the corresponding state processes of (\ref{stateEquation}). Then
	\begin{equation}
		\begin{aligned}
			&\mathcal{H}^{RN}\big(t,\bar{X}(t),\bar{Y}(t),\bar{Z}(t),u,p(t),q(t),P(t),\lambda(t),\xi(t),\eta(t),\Gamma(t)\big)\\
			&\quad \geq\mathcal{H}^{RN}\big(t,\bar{X}(t),\bar{Y}(t),\bar{Z}(t),\bar{u}(t),p(t),q(t),P(t),\lambda(t),\xi(t),\eta(t),\Gamma(t)\big),\\
			&\hspace{5cm} \forall u\in U,\quad a.e.,\quad \mathbb{P}\mbox{-}a.s.,
		\end{aligned}
	\end{equation}
where $\Delta(\cdot)$ satisfies (\ref{algebraEquation}), and $(P(\cdot),Q(\cdot))$, $(\lambda(\cdot),\xi(\cdot),\eta(\cdot))$, $(\Gamma(\cdot),\Lambda(\cdot))$ are solutions of (\ref{second-order-adjoint(PQ)}), (\ref{First-Order-Adjoint(mnk)}) and (\ref{Second-Order_Adjoint(MN)}) respectively.
\end{mylem}

\subsection{Risk-sensitive global maximum principle}

In this section, we give the risk-sensitive global maximum principle, which is the main result of this paper. The proof will be given in the next section.

Consider the state equation (\ref{stateEquation}) combined with the risk-sensitive cost functional (\ref{R-S CostFunctional}). The first-order decoupling field is given as:
\begin{equation}\label{first-order-RS-adjoint(pq)}
	\left\{
	\begin{aligned}
		-dp(t)&=\Big\{ g_x(t)+p(t)g_y(t)+K_1(t)g_z(t)+b^\top_x(t)p(t)+p(t)b^\top_y(t)p(t)\\
		     &\qquad +K_1(t)b^\top_z(t)p(t)+\sigma^\top_x(t)q(t)+p(t)\sigma^\top_y(t)q(t)+K_1(t)\sigma^\top_z(t)q(t)\Big\}\,dt\\
		     &\quad -q(t)\,dW(t),\quad t\in[0,T],\\
		 p(T)&=\phi_x(\bar{X}(T)),
	\end{aligned}
	\right.
\end{equation}
which admits a unique solution $(p(\cdot),q(\cdot))\in L^2_{\mathcal{F}}(\Omega;C([0,T];\mathbb{R}^n))\times L^2_{\mathcal{F}}([0,T];\mathbb{R}^n)$, and $K_1(\cdot)$ is the process defined in (\ref{K_1}). The second-order risk-sensitive decoupling field is:
\begin{equation}\label{second-order-RS-adjoint(PQ)}
	\left\{
	\begin{aligned}
  -dP(t)&=\Big\{G_y(t)P(t)+G_z(t)K_2(t)+\big[b^\top_x(t)+p(t)b^\top_y(t)+K_1(t)b^\top_z(t)\big]P(t)\\
		&\quad\ +P(t)\big[b_x(t)+b_y(t)p^\top(t)+b_z(t)K^\top_1(t)\big]+\big[\sigma^\top_x(t)+p(t)\sigma^\top_y(t)\\
        &\quad\ +K_1(t)\sigma^\top_z(t)\big]Q(t)+Q(t)\big[\sigma_x(t)+\sigma_y(t)p^\top(t)+\sigma_z(t)K^\top_1(t)\big]+\big[\sigma^\top_x(t)\\
		&\quad\ +p(t)\sigma^\top_y(t)+K_1(t)\sigma^\top_z(t)\big]P(t)\big[\sigma_x(t)+\sigma_y(t)p^\top(t)+\sigma_z(t)K^\top_1(t)\big]\\
		&\quad\ +[I_n,p(t),K_1(t)]D^2G(t)[I_n,p(t),K_1(t)]^\top\Big\}\,dt-Q(t)\,dW(t),\quad t\in[0,T],\\
	P(T)&=\phi_{xx}(\bar{X}(T)),
	\end{aligned}
	\right.
\end{equation}	
which admits unique solution $(P(\cdot),Q(\cdot))\in L^2_{\mathcal{F}}([0,T];\mathcal{S}^n)\times L^2_{\mathcal{F}}([0,T];\mathcal{S}^n)$. Further, $K_2(\cdot)$ and $G(\cdot,x,y,z,u,p,q)$ are given in (\ref{K_2}) and (\ref{G()}) respectively.

Similar as the risk-neutral case, we introduce the first-order risk-sensitive adjoint processes $(k(\cdot),m(\cdot),n(\cdot))\in L^2_{\mathcal{F}}(\Omega;C([0,T];\mathbb{R}^n))\times L^2_{\mathcal{F}}(\Omega;C([0,T];\mathbb{R}^n))\times L^2_{\mathcal{F}}([0,T];\mathbb{R}^n)$, which is the unique solution of the following non-coupled FBSDE:
\begin{equation*}\left\{
	\begin{aligned}
   dk(t)&=\Big\{\big[m^\top(t)b_z(t)+l_z(t)+(n^\top(t)+\theta\kappa(t)m^\top(t))\sigma_z(t)\big](1-p^\top(t)\sigma_z(t))^{-1}\\
        &\qquad \times(p^\top(t)\sigma_y(t)-\theta\kappa(t))+m^\top(t)b_y(t)+l_y(t)+(n^\top(t)+\theta\kappa(t)m^\top(t))\sigma_y(t)\\
		&\qquad +k(t)\big[g_y(t)+p^\top(t)b_y(t)+l_y(t)+q^\top(t)\sigma_y(t)+(1-p^\top(t)\sigma_z(t))^{-1}\\
        &\qquad \times g_z(t)p^\top(t)\sigma_y(t)\big]+\theta^2\kappa^2(t)k(t)-\theta\kappa(t)k(t)\big[g_z(t)+p^\top(t)b_z(t)+l_z(t)\\
		&\qquad +q^\top(t)\sigma_z(t)+(1-p^\top(t)\sigma_z(t))^{-1}g_z(t)p^\top(t)\sigma_z(t)\big]\Big\}\,dt\\
		&\quad +\Big\{\big[m^\top(t)b_z(t)+l_z(t)+(n^\top(t)+\theta\kappa(t)m^\top(t))\sigma_z(t)\big](1-p^\top(t)\sigma_z(t))^{-1}\\
		&\qquad\ +k(t)\big[g_z(t)+p^\top(t)b_z(t)+l_z(t)+q^\top(t)\sigma_z(t)+(1-p^\top(t)\sigma_z(t))^{-1}\\
        &\qquad\ \times g_z(t)p^\top(t)\sigma_z(t)\big]-\theta\kappa(t)k(t)\Big\}\,dW(t),\\
	\end{aligned}\right.
\end{equation*}
\begin{equation}\label{First-order Risk-sensitive Adjoint Equation(mnk)}\left\{
	\begin{aligned}
  -dm(t)&=\Big\{\big[b^\top_x(t)+p(t)b^\top_y(t)+K_1(t)b^\top_z(t)\big]m(t)+l^\top_x(t)+p(t)l^\top_y(t)\\
		&\quad\ +K_1(t)l^\top_z(t)+\big[\sigma^\top_x(t)+p(t)\sigma^\top_y(t)+K_1(t)\sigma^\top_z(t)\big](\theta\kappa(t)m(t)+n(t))\\
		&\quad\ +\theta\kappa(t)n(t)\Big\}\,dt-n(t)\,dW(t),\quad t\in[0,T],\\
	k(0)&=f_y(\bar{X}(T),\bar{Y}(0)),\quad m(T)=f_x(\bar{X}(T),\bar{Y}(0)),
	\end{aligned}\right.
\end{equation}
where $(\gamma(\cdot),\kappa(\cdot))\in\mathbb{R}\times\mathbb{R}$ is the unique solution of (\ref{Equation-gamma-kappa}), a BSDE with quadratic generator. The unique solvability of (\ref{First-order Risk-sensitive Adjoint Equation(mnk)}) is discussed in Remark \ref{Remark-unique-solv-first-order-RS(mnk)-trans}.

Next, we introduce the second-order risk-sensitive adjoint processes $(M(\cdot),N(\cdot))\in L^2_{\mathcal{F}}([0,T];\\\mathcal{S}^n)\times L^2_{\mathcal{F}}([0,T];\mathcal{S}^n)$, which satisfies the following matrix-valued linear BSDE:
\begin{equation}\label{Second-order Risk-sensitive Adjoint Equation(MN)}
	\left\{
	\begin{aligned}
  -dM(t)&=\Big\{M(t)\big[b_x(t)+b_y(t)p^\top(t)+b_z(t)K^\top_1(t)\big]\\
        &\qquad +\big[b^\top_x(t)+p(t)b^\top_y(t)+K_1(t)b^\top_z(t)\big]M(t)\\
		&\qquad +\big[\sigma^\top_x(t)+p(t)\sigma^\top_y(t)+K_1(t)\sigma^\top_z(t)\big]\big[M(t)+\theta m(t)m^\top(t)\big]\\
		&\qquad \times\big[\sigma_x(t)+\sigma_y(t)p^\top(t)+\sigma_z(t)K^\top_1(t)\big]\\
		&\qquad +\big[\sigma^\top_x(t)+p(t)\sigma^\top_y(t)+K_1(t)\sigma^\top_z(t)\big]\big[N(t)+\theta\kappa(t)M(t)+\theta m(t)n^\top(t)\big]\\
		&\qquad +\big[N(t)+\theta\kappa(t)M(t)+\theta n(t)m^\top(t)\big]\big[\sigma_x(t)+\sigma_y(t)p^\top(t)+\sigma_z(t)K_1^\top(t)\big]\\
		&\qquad +\big[m^\top(t)b_y(t)+l_y(t)\big]P(t)+\big[m^\top(t)b_z(t)+l_z(t)\big]K_2(t)+\theta\kappa(t)N(t)\\
		&\qquad -\theta n(t)n^\top(t)+m^\top(t)[I_n,p(t),K_1(t)]D^2b(t)[I_n,p(t),K_1(t)]^\top\\
        &\qquad +[I_n,p(t),K_1(t)]D^2l(t)[I_n,p(t),K_1(t)]^\top\\
		&\qquad +\big[n^\top(t)+\theta\kappa(t)m^\top(t)\big][I_n,p(t),K_1(t)]D^2\sigma(t)[I_n,p(t),K_1(t)]^\top\Big\}\,dt\\
        &\quad -N(t)\,dW(t),\quad t\in[0,T],\\
	M(T)&=f_{xx}(\bar{X}(T),\bar{Y}(0)).
	\end{aligned}
	\right.
\end{equation}

Define the risk-sensitive Hamiltonian $\mathcal{H}^{RS}:[0,T]\times\mathbb{R}^n\times\mathbb{R}\times\mathbb{R}\times U\times\mathbb{R}^n\times\mathbb{R}^n\times\mathbb{R}^n\times\mathbb{R}^n\times\mathbb{R}^n\times\mathcal{S}^n\times\mathcal{S}^n\times\mathbb{R}\to\mathbb{R}$:
\begin{equation}\label{Hamiltonian-RS}
	\begin{aligned}
		&\mathcal{H}_\theta^{RS}(t,x,y,z,u,p,q,k,m,n,P,M,\kappa)\\
		&:=\big(m^\top+kp^\top\big)b(t)+l(t)+\big(n^\top+\theta\kappa m^\top+kq^\top\big)\sigma(t)+kg(t)\\
		&\qquad +\frac{1}{2}\big(\sigma(t)-\sigma(t,\bar{X},\bar{Y},\bar{Z},\bar{u})\big)^\top\big(M+\theta mm^\top+kP\big)\big(\sigma(t)-\sigma(t,\bar{X},\bar{Y},\bar{Z},\bar{u})\big),\\
	\end{aligned}
\end{equation}
where $b(t)\equiv b(t,x,y,z+\Delta(t),u)$ for simplification, and $\sigma(t)$, $g(t)$ and $l(t)$ are the same. $\Delta(\cdot)$ is the solution of (\ref{algebraEquation}) corresponding to $u(t)=u$.

Now we give the risk-sensitive global maximum principle, whose proof will be given in the next section.
\begin{mythm}\label{SMP-RS}
	Suppose Assumption \ref{SeparableSubset}, Assumption \ref{AssumSolvableOfStateEquation} and Assumption \ref{q is bounded} hold. Let (\ref{first-order-RS-adjoint(pq)}) and (\ref{algebraEquation}) have unique solutions $(p(\cdot),q(\cdot))\in L^2_{\mathcal{F}}(\Omega;C([0,T],\mathbb{R}))\times L^2_{\mathcal{F}}([0,T];\mathbb{R})$ and $\Delta(\cdot)\in L^2_{\mathcal{F}}([0,T];\mathbb{R}^n)$ respectively. Let $\bar{u}(\cdot)\in\mathcal{U}[0,T]$ be optimal and $(\bar{X}(\cdot),\bar{Y}(\cdot),\bar{Z}(\cdot))$ be the corresponding state processes of (\ref{stateEquation}). Then
	\begin{equation}\label{maximum condition-RS}
		\begin{aligned}
			&\theta\,\mathcal{H}_\theta^{RS}\big(t,\bar{X}(t),\bar{Y}(t),\bar{Z}(t),u,p(t),q(t),k(t),m(t),n(t),P(t),M(t),\kappa(t)\big)\\
			&\geq\theta\,\mathcal{H}_\theta^{RS}\big(t,\bar{X}(t),\bar{Y}(t),\bar{Z}(t),\bar{u}(t),p(t),q(t),k(t),m(t),n(t),P(t),M(t),\kappa(t)\big),\\
			&\hspace{6cm}\forall u\in U,\quad a.e.,\quad \mathbb{P}\mbox{-}a.s.,
		\end{aligned}
	\end{equation}
where $(P(\cdot),Q(\cdot))$, $(k(\cdot),m(\cdot),n(\cdot))$, $(M(\cdot),N(\cdot))$ and $\kappa(\cdot)$ are solutions of (\ref{first-order-RS-adjoint(pq)}), (\ref{second-order-RS-adjoint(PQ)}), (\ref{First-order Risk-sensitive Adjoint Equation(mnk)}), (\ref{Second-order Risk-sensitive Adjoint Equation(MN)}) and (\ref{Equation-gamma-kappa}) respectively.
\end{mythm}

\section{Proof of the main result}

The exponential-of-integral in the cost functional (\ref{R-S CostFunctional}) makes the risk-sensitive problem different from the risk-neutral one. Hence, the risk-neutral global maximum principle can not be used directly, to give the necessary conditions of the optimal control. Inspired by Lim and Zhou \cite{LZ05}, we reformulate the risk-sensitive stochastic optimal control problem (\ref{RS problem}) to the risk-neutral form with the help of a one-dimensional auxiliary process $\tilde{X}(\cdot)$, which is the solution of the following equation:
\begin{equation*}
	\left\{
	\begin{aligned}
		d\tilde{X}(t)&=l(t,X(t),Y(t),Z(t),u(t))\,dt,\quad t\in[0,T],\\
		\tilde{X}(0)&=0,
	\end{aligned}
	\right.
\end{equation*}
where $l$ is defined in (\ref{R-S CostFunctional}). Set
\begin{equation*}
	\begin{aligned}
		\mathbbm{X}(t):=\begin{bmatrix}
			X(t)\\
			\tilde{X}(t)
		\end{bmatrix},\quad\mathbbm{B}(t):=\begin{bmatrix}
			b(t)\\
			l(t)
		\end{bmatrix},\quad\Sigma(t):=\begin{bmatrix}
			\sigma(t)\\
			0
		\end{bmatrix},
	\end{aligned}
\end{equation*}
in which we omit $\mathbbm{X}(t)$, $Y(t)$, $Z(t)$ and $u(t)$ in the above maps for simplify, for $t\in[0,T]$.

Thus we can reformulate the problem (\ref{RS problem}) as:
\begin{equation}\label{StateEquation-RS}
	\left\{
	\begin{aligned}
		d\mathbbm{X}(t)&=\mathbbm{B}(t,\mathbbm{X}(t),Y(t),Z(t),u(t))\,dt+\Sigma(t,\mathbbm{X}(t),Y(t),Z(t),u(t))\,dW(t),\\
		-dY(t)&=g(t,X(t),Y(t),Z(t),u(t))\,dt-Z(t)\,dW(t),\quad t\in[0,T],\\
		\mathbbm{X}(0)&=\begin{bmatrix}
			x^\top_0&0
		\end{bmatrix}^\top\in\mathbb{R}^{n+1},\quad Y(T)=\phi(X(T)),
	\end{aligned}
	\right.
\end{equation}
with the cost functional:
\begin{equation}\label{CostFunctional-RS}
	J^{RS}(u(\cdot))=\mathbb{E}\big[e^{\theta[\,\tilde{X}(T)+f(X(T),Y(0))\,]}\big].
\end{equation}

\subsection{Applying risk-neutral global maximum principle}

Since (\ref{CostFunctional-RS}) only has the terminal term, the results in Section \ref{section-SMP-RN} can be applied now. For the new state equation (\ref{StateEquation-RS}) and cost functional (\ref{CostFunctional-RS}), we recall (\ref{first-order-adjoint(pq)}) as:
\begin{equation}\label{First-Order-adjoint-RS-(p,q)-HighD}
	\left\{
	\begin{aligned}
		-d\mathbbm{p}(t)&=\Big\{g_{\mathbbm{x}}(t)+\mathbbm{p}(t)g_y(t)+\mathbbm{K}_1(t)g_z(t)+\mathbbm{B}^\top_{\mathbbm{x}}(t)\mathbbm{p}(t)+\mathbbm{p}(t)\mathbbm{B}^\top_y(t)\mathbbm{p}(t)\\
		               &\qquad +\mathbbm{K}_1(t)\mathbbm{B}^\top_z(t)\mathbbm{p}(t)+\Sigma^\top_x(t)\mathbbm{q}(t)+\mathbbm{p}(t)\Sigma^\top_y(t)\mathbbm{q}(t)+\mathbbm{K}_1(t)\Sigma^\top_z(t)\mathbbm{q}(t)\Big\}\,dt\\
		               &\quad -\mathbbm{q}(t)\,dW(t),\quad t\in[0,T],\\
		 \mathbbm{p}(T)&=\begin{bmatrix}
			            \phi_x(\bar{X}(T))\\
			            0
		                 \end{bmatrix},
	\end{aligned}
	\right.
\end{equation}
where the $\mathbb{R}^{n+1}$-valued process
$$
\mathbbm{K}_1(t):=(1-\mathbbm{p}^\top(t)\Sigma_z(t))^{-1}\big[\Sigma^\top_{\mathbbm{x}}(t)\mathbbm{p}(t)+\mathbbm{p}(t)\Sigma^\top_y(t)\mathbbm{p}(t)+\mathbbm{q}(t)\big].
$$
Set
\begin{equation*}
	\mathbbm{p}:=\begin{bmatrix}
		p\\
		\tilde{p}
	\end{bmatrix},\quad
\mathbbm{q}:=\begin{bmatrix}
		q\\
		\tilde{q}
	\end{bmatrix},\quad
\mathbbm{K}_1:=\begin{bmatrix}
		K_1\\
		\tilde{K}_1
	\end{bmatrix}.
\end{equation*}
We can deduce that (\ref{First-Order-adjoint-RS-(p,q)-HighD}) has a unique solution $(p(\cdot),q(\cdot))\in L^2_{\mathcal{F}}(\Omega;C([0,T];\mathbb{R}^n))\times L^2_{\mathcal{F}}([0,T];\mathbb{R}^n)$ and $(\tilde{p}(\cdot),\tilde{q}(\cdot))\in L^2_{\mathcal{F}}(\Omega;C([0,T];\mathbb{R}))\times L^2_{\mathcal{F}}([0,T];\mathbb{R})$, and $K_1(\cdot)$ and $\tilde{K}_1(\cdot)$ are processes value in $\mathbb{R}^n$ and $\mathbb{R}$ respectively.

Recall (\ref{second-order-adjoint(PQ)}) as the following:
\begin{equation}\label{Second-Order-adjoint-RS-(P,Q)-HighD}
	\left\{
	\begin{aligned}
		-d\mathbbm{P}(t)&=\Big\{G_y(t)\mathbbm{P}(t)+G_z(t)\mathbbm{K}_2(t)+\big[\mathbbm{B}^\top_{\mathbbm{x}}(t)+\mathbbm{p}(t)\mathbbm{B}^\top_y(t)+\mathbbm{K}_1(t)\mathbbm{B}^\top_z(t)\big]\mathbbm{P}(t)\\
		&\qquad +\mathbbm{P}(t)\big[\mathbbm{B}_{\mathbbm{x}}(t)+\mathbbm{B}_y(t)\mathbbm{p}^\top(t)+\mathbbm{B}_z(t)\mathbbm{K}^\top_1(t)\big]\\
		&\qquad +\big[\Sigma^\top_{\mathbbm{x}}(t)+\mathbbm{p}(t)\Sigma^\top_y(t)+\mathbbm{K}_1(t)\Sigma^\top_z(t)\big]\mathbbm{Q}(t)\\
		&\qquad +\mathbbm{Q}(t)\big[\Sigma_{\mathbbm{x}}(t)+\Sigma_y(t)\mathbbm{p}^\top(t)+\Sigma_z(t)\mathbbm{K}^\top_1(t)\big]\\
		&\qquad +\big[\Sigma^\top_{\mathbbm{x}}(t)+\mathbbm{p}(t)\Sigma^\top_y(t)+\mathbbm{K}_1(t)\Sigma^\top_z(t)\big]\mathbbm{P}(t)\big[\Sigma_{\mathbbm{x}}(t)+\Sigma_y(t)\mathbbm{p}^\top(t)\\
        &\qquad +\Sigma_z(t)\mathbbm{K}^\top_1(t)\big]+[I_{n+1},\mathbbm{p}(t),\mathbbm{K}_1(t)]D^2\mathbbm{G}(t)[I_{n+1},\mathbbm{p}(t),\mathbbm{K}_1(t)]^\top\Big\}\,dt\\
		&\quad -\mathbbm{Q}(t)\,dW(t),\quad t\in[0,T],\\
		\mathbbm{P}(T)&=\begin{bmatrix}
			\phi_{xx}(\bar{X}(T))&0_n  \\
			0_n^\top&0
		\end{bmatrix},
	\end{aligned}
	\right.
\end{equation}
where
$$
\mathbbm{G}(t,\mathbbm{x},y,z,u,\mathbbm{p},\mathbbm{q}):=g(t,\mathbbm{x},y,z,u)+\mathbbm{p}^\top\mathbbm{B}(t,\mathbbm{x},y,z,u)+\mathbbm{q}^\top\Sigma(t,\mathbbm{x},y,z,u),
$$
and the $\mathcal{S}^{n+1}$-valued process
\begin{equation*}
	\begin{aligned} \mathbbm{K}_2(t)&:=(1-\mathbbm{p}^\top(t)\Sigma_z(t))^{-1}\Big\{\mathbbm{p}^\top(t)\Sigma_y(t)\mathbbm{P}(t)+\mathbbm{Q}(t)+\big[\Sigma^\top_{\mathbbm{x}}+\mathbbm{p}\Sigma^\top_y+\mathbbm{K}_1\Sigma^\top_z\big]\mathbbm{P}(t)\\
		        &\qquad +\mathbbm{P}(t)\big[\Sigma_{\mathbbm{x}}+\Sigma_y\mathbbm{p}^\top+\Sigma_z\mathbbm{K}^\top_1\big]
                 +\mathbbm{p}^\top(t)[I_{n+1},\mathbbm{p},\mathbbm{K}_1]D^2\Sigma(t)[I_{n+1},\mathbbm{p},\mathbbm{K}_1]^\top\Big\}.
	\end{aligned}
\end{equation*}
We know that (\ref{Second-Order-adjoint-RS-(P,Q)-HighD}) admits a unique solution $(\mathbbm{P}(\cdot),\mathbbm{Q}(\cdot))\in L^2_{\mathcal{F}}(\Omega;C([0,T];\mathcal{S}^{n+1}))\times L^2_{\mathcal{F}}([0,T];\\\mathcal{S}^{n+1})$.
Set
\begin{equation*}
\mathbbm{P}:=\begin{bmatrix}
		P&P_2  \\
		P_2^\top&P_4
	\end{bmatrix},\quad
\mathbbm{Q}:=\begin{bmatrix}
		Q&Q_2  \\
		Q_2^\top&Q_4
	\end{bmatrix},\quad
\mathbbm{K}_2:=\begin{bmatrix}
		K_2&\tilde{K}_2  \\
		\tilde{K}_2^\top&\tilde{\tilde{K}}_2
	\end{bmatrix},
\end{equation*}
where $P(\cdot)$, $Q(\cdot)$ and $K_2(\cdot)$ are $n$-dimensional matrix-valued processes.% $P_2(\cdot)$, $Q_2(\cdot)$, $P_4(\cdot)$ and $Q_4(\cdot)$ are of suitable dimensional.

After some calculation, we can conclude that:
\begin{equation}\label{tilde p}
	\left\{
	\begin{aligned}
		-d\tilde{p}(t)&=\Big\{\tilde{p}(t)g_y(t)+\tilde{K}_1(t)g_z(t)+\tilde{p}(t)\big[b^\top_y(t)p(t)+l_y(t)\tilde{p}(t)\big]\\
		             &\qquad +\tilde{p}(t)\sigma^\top_y(t)q(t)+\tilde{K}_1(t)\sigma_z^\top(t)q(t)\Big\}\,dt-\tilde{q}(t)\,dW(t),\quad t\in[0,T],\\
		 \tilde{p}(T)&=0,
	\end{aligned}
	\right.
\end{equation}
with
$$
\tilde{K}_1(t)=(1-\mathbbm{p}^\top(t)\Sigma_z(t))^{-1}\big[\tilde{p}(t)\sigma^\top_y(t)p(t)+\tilde{q}(t)\big].
$$
Combining with the unique solvability of BSDE (\ref{tilde p}), we can arrive at $(\tilde{p}(\cdot),\tilde{q}(\cdot))=(0,0)$, $\tilde{K}_1(\cdot)=0$. Using the similar method, we also drive that $(P_2(\cdot),Q_2(\cdot))=(0,0)$, $(P_4(\cdot),Q_4(\cdot))=(0,0)$, $\tilde{K}_2(\cdot)=0$ and $\tilde{\tilde{K}}_2(\cdot)=0$.

Thus, from (\ref{First-Order-adjoint-RS-(p,q)-HighD}) and (\ref{Second-Order-adjoint-RS-(P,Q)-HighD}), we have:
\begin{equation*}
	\left\{
	\begin{aligned}
		-dp(t)&=\Big\{ g_x(t)+p(t)g_y(t)+K_1(t)g_z(t)+b^\top_x(t)p(t)+p(t)b^\top_y(t)p(t)\\
		     &\qquad +K_1(t)b^\top_z(t)p(t)+\sigma^\top_x(t)q(t)+p(t)\sigma^\top_y(t)q(t)+K_1(t)\sigma^\top_z(t)q(t)\Big\}\,dt\\
		     &\quad -q(t)\,dW(t),\quad t\in[0,T],\\
		 p(T)&=\phi_x(\bar{X}(T)),
	\end{aligned}
	\right.
\end{equation*}
and
\begin{equation*}
	\left\{
	\begin{aligned}
		-dP(t)&=\Big\{G_y(t)P(t)+G_z(t)K_2(t)+\big[b^\top_x(t)+p(t)b^\top_y(t)+K_1(t)b^\top_z(t)\big]P(t)\\
		&\qquad +P(t)\big[b_x(t)+b_y(t)p^\top(t)+b_z(t)K^\top_1(t)\big]\\
		&\qquad +\big[\sigma^\top_x(t)+p(t)\sigma^\top_y(t)+K_1(t)\sigma^\top_z(t)\big]Q(t)\\
		&\qquad +Q(t)\big[\sigma_x(t)+\sigma_y(t)p^\top(t)+\sigma_z(t)K^\top_1(t)\big]\\
		&\qquad +\big[\sigma^\top_x(t)+p(t)\sigma^\top_y(t)+K_1(t)\sigma^\top_z(t)\big]P(t)\big[\sigma_x(t)+\sigma_y(t)p^\top(t)+\sigma_z(t)K^\top_1(t)\big]\\
		&\qquad +[I_n,p(t),K_1(t)]D^2G(t)[I_n,p(t),K_1(t)]^\top\Big\}\,dt-Q(t)\,dW(t),\quad t\in[0,T],\\
		P(T)&=\phi_{xx}(\bar{X}(T)),
	\end{aligned}
	\right.
\end{equation*}
which coincide with (\ref{first-order-RS-adjoint(pq)}) and (\ref{second-order-RS-adjoint(PQ)}) respectively.

Next, we set
\begin{equation*}
\tilde{\mathbbm{m}}:=\begin{bmatrix}
		\tilde{m}_n\\
		\tilde{m}_2
	\end{bmatrix}_{n+1},\quad
\tilde{\mathbbm{n}}:=\begin{bmatrix}
		\tilde{n}_n\\
		\tilde{n}_2
	\end{bmatrix}_{n+1},\quad
\tilde{\mathbbm{M}}:=\begin{bmatrix}
		\tilde{M}_{n\times n}&\tilde{M}_2\\
		\tilde{M}^\top_2&\tilde{M}_4
	\end{bmatrix}_{(n+1)^2},\quad
\tilde{\mathbbm{N}}:=\begin{bmatrix}
		\tilde{N}_{n\times n}&\tilde{N}_2\\
		\tilde{N}^\top_2&\tilde{N}_4
	\end{bmatrix}_{(n+1)^2}.
\end{equation*}
In the above, we use the subscript $n$, $n\times n$ and so on, to denote the dimensions of related matrix-valued processes.
%where $(\tilde{m},\tilde{n})\in\mathbb{R}^n\times\mathbb{R}^n$, $(\tilde{m}_2,\tilde{n}_2)\in\mathbb{R}\times\mathbb{R}$, $(\tilde{M},\tilde{N})\in\mathcal{S}^n\times\mathcal{S}^n$, $(\tilde{M}_2,\tilde{N}_2)\in\mathbb{R}^n\times\mathbb{R}^n$ and $(\tilde{M}_4,\tilde{N}_4)\in\mathbb{R}\times\mathbb{R}$

Recall (\ref{First-Order-Adjoint(mnk)}) as:
\begin{equation}\label{First-Order-Adjoint(mnk)-HighD}
	\left\{
	\begin{aligned}
d\tilde{k}(t)&=\Big\{\big[\tilde{\mathbbm{m}}^\top(t)\mathbbm{B}_z(t)+\tilde{\mathbbm{n}}^\top(t)\Sigma_z(t)\big](1-\mathbbm{p}^\top(t)\Sigma_z(t))^{-1}\mathbbm{p}^\top(t)\Sigma_y(t)
                     +\tilde{\mathbbm{m}}^\top(t)\mathbbm{B}_y(t)\\
		     &\qquad +\tilde{\mathbbm{n}}^\top(t)\Sigma_y(t)+\tilde{k}(t)\big[\mathbbm{G}_y(t)+(1-\mathbbm{p}^\top(t)\Sigma_z(t))^{-1}g_z(t)\mathbbm{p}^\top(t)\Sigma_y(t)\big]\Big\}\,dt\\
		     &\quad +\Big\{\big[\tilde{\mathbbm{m}}^\top(t)\mathbbm{B}_z(t)+\tilde{\mathbbm{n}}^\top(t)\Sigma_z(t)\big](1-\mathbbm{p}^\top(t)\Sigma_z(t))^{-1}\mathbbm{p}^\top(t)\Sigma_z(t)\\
		     &\qquad\ +\tilde{\mathbbm{m}}^\top(t)\mathbbm{B}_z(t)+\tilde{\mathbbm{n}}^\top(t)\Sigma_z(t)\\
		     &\qquad\ +\tilde{k}(t)\big[\mathbbm{G}_z(t)+(1-\mathbbm{p}^\top(t)\Sigma_z(t))^{-1}g_z(t)\mathbbm{p}^\top(t)\Sigma_z(t)\big]\Big\}\,dW(t),\\
             %%%%
-d\tilde{\mathbbm{m}}(t)&=\Big\{\big[\mathbbm{B}^\top_{\mathbbm{x}}(t)+\mathbbm{p}(t)\mathbbm{B}^\top_y(t)+\mathbbm{K}_1(t)\mathbbm{B}^\top_z(t)\big]\tilde{\mathbbm{m}}(t)\\
		     &\qquad +\big[\Sigma^\top_{\mathbbm{x}}(t)+\mathbbm{p}(t)\Sigma^\top_y(t)+\mathbbm{K}_1(t)\Sigma^\top_z(t)\big]\tilde{\mathbbm{n}}(t)\Big\}\,dt-\tilde{\mathbbm{n}}(t)\,dW(t),\quad t\in[0,T],\\
		     %%%%
		     \tilde{k}(0)&=\theta e^{\theta[\,\bar{\tilde{X}}(T)+f(\bar{X}(T),\bar{Y}(0))\,]}f_y(\bar{X}(T),\bar{Y}(0)),\\
             %%%%
\tilde{\mathbbm{m}}(T)&=\theta e^{\theta[\,\bar{\tilde{X}}(T)+f(\bar{X}(T),\bar{Y}(0))\,]}\begin{bmatrix}
			f_x(\bar{X}(T),\bar{Y}(0))\\
			1
		\end{bmatrix},\\
	\end{aligned}
	\right.
\end{equation}
which admits a unique solution $(\tilde{k}(\cdot),\tilde{\mathbbm{m}}(\cdot),\tilde{\mathbbm{n}}(\cdot))\in L^2_{\mathcal{F}}(\Omega;C([0,T];\mathbb{R}^{n+1}))\times L^2_{\mathcal{F}}(\Omega;C([0,T];\\\mathbb{R}^{n+1}))\times L^2_{\mathcal{F}}([0,T];\mathbb{R}^{n+1})$, and recall (\ref{Second-Order_Adjoint(MN)}) as:
\begin{equation}\label{Second-Order_Adjoint(MN)-HighD}
	\left\{
	\begin{aligned}
-d\tilde{\mathbbm{M}}(t)&=\Big\{\tilde{\mathbbm{M}}(t)\big[\mathbbm{B}_{\mathbbm{x}}(t)+\mathbbm{B}_y(t)\mathbbm{p}^\top(t)+\mathbbm{B}_z(t)\mathbbm{K}^\top_1(t)\big]\\
        &\qquad +\big[\mathbbm{B}^\top_{\mathbbm{x}}(t)+\mathbbm{p}(t)\mathbbm{B}^\top_y(t)+\mathbbm{K}_1(t)\mathbbm{B}^\top_z(t)\big]\tilde{\mathbbm{M}}(t)\\
	    &\qquad +\big[\Sigma^\top_{\mathbbm{x}}(t)+\mathbbm{p}(t)\Sigma^\top_y(t)+\mathbbm{K}_1(t)\Sigma^\top_z(t)\big]\tilde{\mathbbm{M}}(t)
         \big[\Sigma_{\mathbbm{x}}(t)+\Sigma_y(t)\mathbbm{p}^\top(t)\\
		&\qquad +\Sigma_z(t)\mathbbm{K}^\top_1(t)\big]+\tilde{\mathbbm{m}}^\top(t)\big[\mathbbm{B}_y(t)\mathbbm{P}(t)+\mathbbm{B}_z(t)\mathbbm{K}_2(t)\\
        &\qquad +[I_{n+1},\mathbbm{p}(t),\mathbbm{K}_1(t)]D^2\mathbbm{B}(t)[I_{n+1},\mathbbm{p}(t),\mathbbm{K}_1(t)]^\top\big]+\tilde{\mathbbm{n}}^\top(t)\big[\Sigma_y(t)\mathbbm{P}(t)\\
		&\qquad +\Sigma_z(t)\mathbbm{K}_2(t)+[I_{n+1},\mathbbm{p}(t),\mathbbm{K}_1(t)]D^2\Sigma(t)[I_{n+1},\mathbbm{p}(t),\mathbbm{K}_1(t)]^\top\big]\Big\}\,dt\\
		&\quad -\tilde{\mathbbm{N}}(t)\,dW(t),\quad t\in[0,T],\\
\tilde{\mathbbm{M}}(T)&=\theta e^{\theta[\,\bar{\tilde{X}}(T)+f(\bar{X}(T),\bar{Y}(0))\,]}\begin{bmatrix}
			\bar{f}_{xx}+\theta \bar{f}_x\bar{f}_x^{\top}&\theta \bar{f}_x  \\
			\theta \bar{f}_x^{\top}&\theta
		\end{bmatrix},
	\end{aligned}
	\right.
\end{equation}
which admits a unique solution $(\tilde{\mathbbm{M}}(\cdot),\tilde{\mathbbm{N}}(\cdot))\in L^2_{\mathcal{F}}(\Omega;C([0,T];\mathcal{S}^{n+1}))\times L^2_{\mathcal{F}}([0,T];\mathcal{S}^{n+1})$.

\subsection{Non-linear transformations}

We can get the risk-sensitive global maximum principle from Lemma \ref{SMP-MayerType} directly, since we have established the first- and second-order adjoint equations we need. However, it is not an ideal one since the adjoint variables involve auxiliary and unnecessary components. It is to say that $\tilde{\mathbbm{m}}(\cdot)$, $\tilde{\mathbbm{n}}(\cdot)$, $\tilde{\mathbbm{M}}(\cdot)$ and $\tilde{\mathbbm{N}}(\cdot)$ are of dimension $n+1$ whereas state $X(\cdot)$ is of dimension $n$.

In this section, we introduce some non-linear transformations to eliminate the additional parts of $(\tilde{\mathbbm{m}}(\cdot),\tilde{\mathbbm{n}}(\cdot))$ and $(\tilde{\mathbbm{M}}(\cdot),\tilde{\mathbbm{N}}(\cdot))$, which finally lead to Theorem \ref{SMP-RS}.

\subsubsection{Transformation of the first-order processes}

Form (\ref{First-Order-Adjoint(mnk)-HighD}), we can easily conclude that:
\begin{equation}
	\left\{
	\begin{aligned}
		d\tilde{m}_2(t)&=\tilde{n}_2(t)\,dW(t),\quad t\in[0,T],\\
		\tilde{m}_2(T)&=\theta e^{\theta\big[\,\bar{\tilde{X}}(T)+f(\bar{X}(T),\bar{Y}(0))\,\big]}.
	\end{aligned}
	\right.
\end{equation}
As in Moon \cite{Moon20}, its $\mathcal{F}_t$-adapted solution can be written as:
$$\tilde{m}_2(t)=\mathbb{E}\Big[\theta e^{\theta\big[\,\bar{\tilde{X}}(T)+f(\bar{X}(T),\bar{Y}(0))\,\big]}\,\Big|\,\mathcal{F}_t\Big]:=\theta\zeta(t),$$
where $\zeta(\cdot)$ is a martingale values in $\mathbb{R}$ such that $\zeta(t)>0$, $\mathbb{P}$-a.s., for $t\in[0,T]$.

Introducing the following non-linear transformation
\begin{equation}\label{transformation first-order}
	\begin{aligned}
\mathbbm{m}(t)=\frac{\tilde{\mathbbm{m}}(t)}{\theta\zeta(t)}:=\begin{bmatrix}
	m(t)\\
	m_2(t)
\end{bmatrix},
	\end{aligned}
\end{equation}
which makes the terminal condition of (\ref{First-Order-Adjoint(mnk)-HighD}) become
$$
\mathbbm{m}(T)=\begin{bmatrix}
	f_x(\bar{X}(T),\bar{Y}(0))\\
	1
\end{bmatrix}.
$$

Before we can apply It\^o's formula to $\mathbbm{m}(\cdot)$, we first have to compute $d\frac{1}{\zeta(\cdot)}$, in other words, $d\zeta(\cdot)$. In fact, we have
\begin{equation*}
	\begin{aligned}
		\zeta(t)&=\mathbb{E}\Big[e^{\theta\big[\,\bar{\tilde{X}}(T)+f(\bar{X}(T),\bar{Y}(0))\,\big]}\,\Big|\,\mathcal{F}_t\Big]\\
		&=e^{\theta\int_0^tl(s,\bar{X},\bar{Y},\bar{Z},\bar{u})\,ds}\times\mathbb{E}\Big[e^{\theta\big[\,\int_t^Tl(s,\bar{X},\bar{Y},\bar{Z},\bar{u})\,ds+f(\bar{X}(T),\bar{Y}(0))\,\big]}\,\Big|\,\mathcal{F}_t\Big]
		:=e^{\theta\int_0^tl(s,\bar{X},\bar{Y},\bar{Z},\bar{u})\,ds}\bar{\zeta}(t).
	\end{aligned}
\end{equation*}

Assuming that $\bar{\zeta}(t)=e^{\theta\gamma(t)}$ and
\begin{equation*}
	\left\{
	\begin{aligned}
		d\gamma(t)&=\mathcal{A}(t)\,dt+\kappa(t)\,dW(t),\quad t\in[0,T],\\
		\gamma(T)&=f(\bar{X}(T),\bar{Y}(0)),
	\end{aligned}
	\right.
\end{equation*}
where $\mathcal{A}(\cdot)$ is to be determined. Define $\bar{\gamma}(t):=e^{\theta\gamma(t)}$, we now choose suitable $\mathcal{A}(t)$ such that $\bar{\gamma}(t)=\bar{\zeta}(t)$. By It\^o's formula,
\begin{equation*}
		d\bar{\gamma}(t)=\theta\bar{\gamma}(t)\Big\{\Big[\mathcal{A}(t)+\frac{\theta}{2}\kappa^2(t)\Big]\,dt+\kappa(t)\,dW(t)\Big\}.
\end{equation*}
Set $\rho(t):=e^{\theta\int_{0}^{t}l(s,\bar{X},\bar{Y},\bar{Z},\bar{u})\,ds}$. Then we can easily deduce that
$$
d\rho(t)\bar{\gamma}(t)=\theta\rho(t)\bar{\gamma}(t)\Big\{\Big[\mathcal{A}(t)+l(t,\bar{X},\bar{Y},\bar{Z},\bar{u})+\frac{\theta}{2}\kappa^2(t)\Big]\,dt+\kappa(t)\,dW(t)\Big\}.
$$
Integral from $t$ to $T$, we have
\begin{equation*}
	\begin{aligned}
		\rho(T)\bar{\gamma}(T)-\rho(t)\bar{\gamma}(t)&=\int_t^T\theta\rho(s)\bar{\gamma}(s)\Big[\mathcal{A}(s)+l(s,\bar{X},\bar{Y},\bar{Z},\bar{u})+\frac{\theta}{2}\kappa^2(s)\Big]\,ds\\
		&\quad+\int_{t}^{T}\theta\rho(s)\bar{\gamma}(s)\kappa(s)\,dW(s).
	\end{aligned}
\end{equation*}
Notice that
$\rho(T)\bar{\gamma}(T)=e^{\theta\int_0^Tl(t,\bar{X},\bar{Y},\bar{Z},\bar{u})\,dt+\theta f(\bar{X}(T),\bar{Y}(0))}$, we arrive at
\begin{equation*}
	\begin{aligned}
		\bar{\gamma}(t)&=e^{\theta\int_t^Tl(s,\bar{X},\bar{Y},\bar{Z},\bar{u})\,ds+\theta f(\bar{X}(T),\bar{Y}(0))}-\rho^{-1}(t)\int_t^T\theta\rho(s)\bar{\gamma}(s)\kappa(s)\,dW(s)\\
		&\quad-\rho^{-1}(t)\int_t^T\theta\rho(s)\bar{\gamma}(s)\Big[\mathcal{A}(s)+l(s,\bar{X},\bar{Y},\bar{Z},\bar{u})+\frac{\theta}{2}\kappa^2(s)\Big]\,ds.
	\end{aligned}
\end{equation*}
Since $\bar{\gamma}(t)=e^{\theta\gamma(t)}$ is $\mathcal{F}_t$-adapted, then we have
\begin{equation*}
	\begin{aligned}
		\bar{\gamma}(t)&=\mathbb{E}\Big[e^{\theta\int_t^Tl(s,\bar{X},\bar{Y},\bar{Z},\bar{u})\,ds+\theta f(\bar{X}(T),\bar{Y}(0))}\,\Big|\,\mathcal{F}_t\Big]\\
		&\quad-\mathbb{E}\bigg[\rho^{-1}(t)\int_t^T\theta\rho(s)\bar{\gamma}(s)\big[\mathcal{A}(s)+l(s,\bar{X},\bar{Y},\bar{Z},\bar{u})+\frac{\theta}{2}\kappa^2(s)\big]\,ds\,\Big|\,\mathcal{F}_t\bigg]\\
		&=\bar{\zeta}(t)-\mathbb{E}\bigg[\rho^{-1}(t)\int_{t}^{T}\theta\rho(s)\bar{\gamma}(s)\big[\mathcal{A}(s)+l(s,\bar{X},\bar{Y},\bar{Z},\bar{u})+\frac{\theta}{2}\kappa^2(s)\big]\,ds\,\Big|\,\mathcal{F}_t\bigg].
	\end{aligned}
\end{equation*}
Set $\mathcal{A}(t):=-\big[l(t,\bar{X},\bar{Y},\bar{Z},\bar{u})+\frac{\theta}{2}\kappa^2(t)\big]$, then we finally conclude that $\bar{\gamma}(t)=\bar{\zeta}(t)=e^{\theta\gamma(t)}$, where $\gamma(\cdot)$ is the solution of the following quadratic growth BSDE, whose unique solvability can be found in El Karoui and Hamadene \cite{EH03} or Zhang \cite{Zhang17}:
\begin{equation}\label{Equation-gamma-kappa}
	\left\{
	\begin{aligned}
		d\gamma(t)&=-\Big[l(t,\bar{X}(t),\bar{Y}(t),\bar{Z}(t),\bar{u}(t))+\frac{\theta}{2}\kappa^2(t)\Big]\,dt+\kappa(t)\,dW(t),\quad t\in[0,T],\\
		\gamma(T)&=f(\bar{X}(T),\bar{Y}(0)).
	\end{aligned}
	\right.
\end{equation}

From It\^o's formula, we arrive at
\begin{equation}\label{Equation-zeta-kappa}
	\left\{
	\begin{aligned}
		d\zeta(t)&=\theta\zeta(t)\kappa(t)\,dW(t),\quad t\in[0,T],\\
		\zeta(T)&=e^{\theta\big[\,\bar{\tilde{X}}(T)+f(\bar{X}(T),\bar{Y}(0))\,\big]},
	\end{aligned}
	\right.
\end{equation}
and
\begin{equation*}
	d\frac{1}{\zeta(t)}=\frac{\theta^2\kappa^2(t)}{\zeta(t)}\,dt-\frac{\theta\kappa(t)}{\zeta(t)}\,dW(t).
\end{equation*}
Hence, by using It\^o's formula again, we get
\begin{equation}\label{First-order-RS-adjoint(mn)-HighD-trans}
\left\{
	\begin{aligned}
		-d\mathbbm{m}(t)&=\Big\{\big[\mathbbm{B}^\top_{\mathbbm{x}}(t)+\mathbbm{p}(t)\mathbbm{B}^\top_y(t)+\mathbbm{K}_1(t)\mathbbm{B}^\top_z(t)\big]\mathbbm{m}(t)
         +\big[\Sigma^\top_{\mathbbm{x}}(t)+\mathbbm{p}(t)\Sigma^\top_y(t)+\mathbbm{K}_1(t)\\
		&\qquad \times\Sigma^\top_z(t)\big](\theta\kappa(t)\mathbbm{m}(t)+\mathbbm{n}(t))+\theta\kappa(t)\mathbbm{n}(t)\Big\}\,dt-\mathbbm{n}(t)\,dW(t),\ t\in[0,T],\\
		\mathbbm{m}(T)&=\begin{bmatrix}
			f_x(\bar{X}(T),\bar{Y}(0))\\
			1
		\end{bmatrix},
	\end{aligned}
\right.
\end{equation}
where
\begin{equation}\label{transformation first-order---}
	\mathbbm{m}(t)=\frac{\tilde{\mathbbm{m}}(t)}{\theta\zeta(t)}:=\begin{bmatrix}
		m(t)\\
		m_2(t)
	\end{bmatrix},\quad\mathbbm{n}(t)=\frac{\tilde{\mathbbm{n}}(t)}{\theta\zeta(t)}-\theta\kappa(t)\mathbbm{m}(t):=\begin{bmatrix}
		n(t)\\
		n_2(t)
	\end{bmatrix}.
\end{equation}

As before, consider the scalar components $(m_2(\cdot),n_2(\cdot))$, we have
\begin{equation*}
	\left\{
	\begin{aligned}
		-dm_2(t)&=\theta\kappa(t)n_2(t)\,dt-n_2(t)\,dW(t),\quad t\in[0,T],\\
		m_2(T)&=1.
	\end{aligned}
	\right.
\end{equation*}
Then the unique solvability of BSDEs guarantees that $(m_2(t),n_2(t))=(1,0)$, $\mathbb{P}$-a.s., $t\in[0,T]$.

Therefore, from (\ref{First-order-RS-adjoint(mn)-HighD-trans}), $(m(\cdot),n(\cdot))$ satisfies
\begin{equation*}
	\left\{
	\begin{aligned}
		-dm(t)&=\Big\{\big[b^\top_x(t)+p(t)b^\top_y(t)+K_1(t)b^\top_z(t)\big]m(t)+\big[l^\top_x(t)+p(t)l^\top_y(t)+K_1(t)l^\top_z(t)\big]\\
		&\qquad +\big[\sigma^\top_x(t)+p(t)\sigma^\top_y(t)+K_1(t)\sigma^\top_z(t)\big](\theta\kappa(t)m(t)+n(t))\\
		&\qquad +\theta\kappa(t)n(t)\Big\}\,dt-n(t)\,dW(t),\quad t\in[0,T],\\
		m(T)&=f_x(\bar{X}(T),\bar{Y}(0)),
	\end{aligned}
	\right.
\end{equation*}
which coincides with the backward equation of (\ref{First-order Risk-sensitive Adjoint Equation(mnk)}).

Moreover, set $k(t)=\frac{\tilde{k}(t)}{\theta\zeta(t)}$. Then by It\^o's formula, we deduce that
\begin{equation*}
	\left\{
	\begin{aligned}
   dk(t)&=\Big\{\big[m^\top(t)b_z(t)+l_z(t)+(n^\top(t)+\theta\kappa(t)m^\top(t))\sigma_z(t)\big](1-p^\top(t)\sigma_z(t))^{-1}p^\top(t)\sigma_y(t)\\
		&\qquad +m^\top(t)b_y(t)+l_y(t)+(n^\top(t)+\theta\kappa(t)m^\top(t))\sigma_y(t)+k(t)\big[g_y(t)\\
		&\qquad +p^\top(t)b_y(t)+l_y(t)+q^\top(t)\sigma_y(t)+(1-p^\top(t)\sigma_z(t))^{-1}g_z(t)p^\top(t)\sigma_y(t)\big]\\
		&\qquad +\theta^2\kappa^2(t)k(t)-\theta\kappa(t)\big[m^\top(t)b_z(t)+l_z(t)+(n^\top(t)+\theta\kappa(t)m^\top(t))\sigma_z(t)\big]\\
		&\qquad -\theta\kappa(t)\big[m^\top(t)b_z(t)+l_z(t)+(n^\top(t)+\theta\kappa(t)m^\top(t))\sigma_z(t)\big]\\
        &\qquad \times(1-p^\top(t)\sigma_z(t))^{-1}p^\top(t)\sigma_z(t)-\theta\kappa(t)k(t)\big[g_z(t)+p^\top(t)b_z(t)+l_z(t)\\
		&\qquad +q^\top(t)\sigma_z(t)+(1-p^\top(t)\sigma_z(t))^{-1}g_z(t)p^\top(t)\sigma_z(t)\big]\Big\}\,dt\\
		&\quad +\Big\{\big[m^\top(t)b_z(t)+l_z(t)+(n^\top(t)+\theta\kappa(t)m^\top(t))\sigma_z(t)\big]\\
        &\qquad\ \times(1-p^\top(t)\sigma_z(t))^{-1}p^\top(t)\sigma_z(t)+k(t)\big[g_z(t)+p^\top(t)b_z(t)+l_z(t)\\
		&\qquad\ +q^\top(t)\sigma_z(t)+(1-p^\top(t)\sigma_z(t))^{-1}g_z(t)p^\top(t)\sigma_z(t)\big]+m^\top(t)b_z(t)\\
		&\qquad\ +l_z(t)+(n^\top(t)+\theta\kappa(t)m^\top(t))\sigma_z(t)-\theta\kappa(t)k(t)\Big\}\,dW(t),\quad t\in[0,T],\\
	k(0)&=f_y(\bar{X}(T),\bar{Y}(0)),
	\end{aligned}
	\right.
\end{equation*}
which coincides with the forward equation of (\ref{First-order Risk-sensitive Adjoint Equation(mnk)}).

\begin{Remark}\label{Remark-unique-solv-first-order-RS(mnk)-trans}
	(\ref{First-Order-Adjoint(mnk)-HighD}) is a non-coupled FBSDE. With the help of classical BSDE theory, (\ref{First-Order-Adjoint(mnk)-HighD}) has a unique solution $(\tilde{\mathbbm{m}}(\cdot),\tilde{\mathbbm{n}}(\cdot))$ since it is not coupled with the forward term $\tilde{k}(\cdot)$. Hence, the unique solvability of (\ref{First-order Risk-sensitive Adjoint Equation(mnk)}) is guaranteed by the unique solvability of (\ref{First-Order-Adjoint(mnk)-HighD}) and (\ref{Equation-zeta-kappa}).
\end{Remark}

\subsubsection{Transformation of the second-order processes and maximum condition}

In this section, we consider the transformation of the second-order processes in (\ref{Second-Order_Adjoint(MN)-HighD}).

Similar as the previous section, we introduce the transformation
\begin{equation}\label{transformation second-order}
	\mathbbm{M}(t)=\frac{\tilde{\mathbbm{M}}(t)}{\theta\zeta(t)}-\theta\mathbbm{m}(t)\mathbbm{m}^\top(t):=\begin{bmatrix}
		M(t)&M_2(t)  \\
		M_2^\top(t)& M_4(t)
	\end{bmatrix},
\end{equation}
and
\begin{equation}\label{transformation second-order---}
	\begin{aligned}
		\mathbbm{N}(t)&=\frac{\tilde{\mathbbm{N}}(t)}{\theta\zeta(t)}-\theta\kappa(t)\big(\mathbbm{M}(t)+\theta\mathbbm{m}(t)\mathbbm{m}^\top(t)\big)
         -\theta\mathbbm{m}(t)\mathbbm{n}^\top(t)-\theta \mathbbm{n}(t)\mathbbm{m}^\top(t)\\
		&:=\begin{bmatrix}
			N(t)&N_2(t)  \\
			N_2^\top(t)& N_4(t)
		\end{bmatrix}.
	\end{aligned}
\end{equation}

By using It\^o's formula, we have
\begin{equation}\label{Second-order-RS-adjoint-(MN)-HighD-trans}
	\left\{
	\begin{aligned}
-d\mathbbm{M}(t)&=\Big\{\mathbbm{M}(t)\big[\mathbbm{B}_{\mathbbm{x}}(t)+\mathbbm{B}_y(t)\mathbbm{p}^\top(t)+\mathbbm{B}_z(t)\mathbbm{K}^\top_1(t)\big]
                 +\big[\mathbbm{B}^\top_{\mathbbm{x}}(t)+\mathbbm{p}(t)\mathbbm{B}^\top_y(t)\\
                &\qquad +\mathbbm{K}_1(t)\mathbbm{B}^\top_z(t)\big]\mathbbm{M}(t)+\big[\Sigma^\top_{\mathbbm{x}}(t)+\mathbbm{p}(t)\Sigma^\top_y(t)+\mathbbm{K}_1(t)\Sigma^\top_z(t)\big]\big[\mathbbm{M}(t)\\
		        &\qquad +\theta\mathbbm{m}(t)\mathbbm{m}^\top(t)\big]\big[\Sigma_{\mathbbm{x}}(t)+\Sigma_y(t)\mathbbm{p}^\top(t)+\Sigma_z(t)\mathbbm{K}^\top_1(t)\big]\\
		        &\qquad +\big[\Sigma^\top_{\mathbbm{x}}(t)+\mathbbm{p}(t)\Sigma^\top_y(t)+\mathbbm{K}_1(t)\Sigma^\top_z(t)\big]\big[\mathbbm{N}(t)+\theta\kappa(t)\mathbbm{M}(t)+\theta\mathbbm{m}(t)\mathbbm{n}^\top(t)\big]\\
		        &\qquad +\big[\mathbbm{N}(t)+\theta\kappa(t)\mathbbm{M}(t)+\theta\mathbbm{n}(t)\mathbbm{m}^\top(t)\big]\big[\Sigma_{\mathbbm{x}}(t)+\Sigma_y(t)\mathbbm{p}^\top(t)+\Sigma_z(t)\mathbbm{K}_1^\top(t)\big]\\
		        &\qquad +\mathbbm{m}^\top(t)\big[\mathbbm{B}_y(t)\mathbbm{P}(t)+\mathbbm{B}_z(t)\mathbbm{K}_2(t)\\
                &\qquad\quad +[I_{n+1},\mathbbm{p}(t),\mathbbm{K}_1(t)]D^2\mathbbm{B}(t)[I_{n+1},\mathbbm{p}(t),\mathbbm{K}_1(t)]^\top\big]\\		
                &\qquad +(\mathbbm{n}(t)+\theta\kappa(t)\mathbbm{m}(t))^\top\big[\Sigma_y(t)\mathbbm{P}(t)+\Sigma_z(t)\mathbbm{K}_2(t)\\
                &\qquad\quad +[I_{n+1},\mathbbm{p}(t),\mathbbm{K}_1(t)]D^2\Sigma(t)[I_{n+1},\mathbbm{p}(t),\mathbbm{K}_1(t)]^\top\big]\\
		        &\qquad +\theta\kappa(t)\mathbbm{N}(t)-\theta\mathbbm{n}(t)\mathbbm{n}^\top(t)\Big\}\,dt-\mathbbm{N}(t)\,dW(t),\quad t\in[0,T],\\
  \mathbbm{M}(T)&=\begin{bmatrix}
			f_{xx}(\bar{X}(T),\bar{Y}(0))&0_n \\
			0_n^\top&0
		\end{bmatrix},
	\end{aligned}
	\right.
\end{equation}
which is uniquely solvable as discussed in Remark \ref{Remark-unique-solv-first-order-RS(mnk)-trans}.

Consider the matrix form here, we can show that
\begin{equation*}
	\left\{
	\begin{aligned}
		-dM_4(t)&=\theta\kappa(t)N_4(t)\,dt-N_4(t)\,dW(t),\quad t\in[0,T],\\
		  M_4(T)&=0,
	\end{aligned}
	\right.
\end{equation*}
which has the unique solution $(M_4(\cdot),N_4(\cdot))=(0,0)$.

Further, the equation of $(M_2(\cdot),N_2(\cdot))$ is the following:
\begin{equation*}
	\left\{
	\begin{aligned}
		-dM_2(t)&=\Big\{\big[b_x^\top(t)+p(t)b_y^\top(t)+K_1(t)b_z^\top(t)\big]M_2(t)+\big[\sigma_x^\top(t)+p(t)\sigma_y^\top(t)\\
		        &\quad +K_1(t)\sigma_z^\top(t)\big]\big[N_2(t)+\theta\kappa(t)M_2(t)\big]+\theta\kappa(t)N_2(t)\Big\}\,dt-N_2(t)\,dW(t),\quad t\in[0,T],\\
		  M_2(T)&=0,
	\end{aligned}
	\right.
\end{equation*}
which also implies that $(M_2(\cdot),N_2(\cdot))=(0,0)$ is the unique solution.

Hence, $(M(\cdot),N(\cdot))$ is the unique solution given by
\begin{equation*}
	\left\{
	\begin{aligned}
  -dM(t)&=\Big\{M(t)\big[b_x(t)+b_y(t)p^\top(t)+b_z(t)K^\top_1(t)\big]\\
        &\qquad +\big[b^\top_x(t)+p(t)b^\top_y(t)+K_1(t)b^\top_z(t)\big]M(t)\\
		&\qquad +\big[\sigma^\top_x(t)+p(t)\sigma^\top_y(t)+K_1(t)\sigma^\top_z(t)\big]\big[M(t)+\theta m(t)m^\top(t)\big]\\
		&\qquad\quad \times\big[\sigma_x(t)+\sigma_y(t)p^\top(t)+\sigma_z(t)K^\top_1(t)\big]\\
		&\qquad +\big[\sigma^\top_x(t)+p(t)\sigma^\top_y(t)+K_1(t)\sigma^\top_z(t)\big]\big[N(t)+\theta\kappa(t)M(t)+\theta m(t)n^\top(t)\big]\\
		&\qquad +\big[N(t)+\theta\kappa(t)M(t)+\theta n(t)m^\top(t)\big]\big[\sigma_x(t)+\sigma_y(t)p^\top(t)+\sigma_z(t)K_1^\top(t)\big]\\
		&\qquad +\big[m^\top(t)b_y(t)+l_y(t)\big]P(t)+\big[m^\top(t)b_z(t)+l_z(t)\big]K_2(t)\\
        &\qquad +\theta\kappa(t)N(t)-\theta n(t)n^\top(t)+m^\top(t)[I_n,p(t),K_1(t)]D^2b(t)[I_n,p(t),K_1(t)]^\top\\
		&\qquad +[I_n,p(t),K_1(t)]D^2l(t)[I_n,p(t),K_1(t)]^\top\\
		&\qquad +(n(t)+\theta\kappa(t)m(t))^\top[I_n,p(t),K_1(t)]D^2\sigma(t)[I_n,p(t),K_1(t)]^\top\Big\}\,dt\\
        &\quad -N(t)\,dW(t),\quad t\in[0,T],\\
	M(T)&=f_{xx}(\bar{X}(T),\bar{Y}(0)),
	\end{aligned}
	\right.
\end{equation*}
which is equivalent to (\ref{Second-order Risk-sensitive Adjoint Equation(MN)}).

Now, we consider the  maximum condition. Recall the risk-neutral Hamiltonian function (\ref{Hamiltonian-RN}) and combine with transformations given above, we have
\begin{equation}
	\begin{aligned}
		&\tilde{\mathcal{H}}^{RN}(t,x,y,z,u,\mathbbm{p},\mathbbm{q},\mathbbm{P},\tilde{k},\tilde{\mathbbm{m}},\tilde{\mathbbm{n}},\tilde{\mathbbm{M}})\\
		&:=(\tilde{\mathbbm{m}}^\top+\tilde{k}\mathbbm{p}^\top)\mathbbm{B}(t,x,y,z+\Delta,u)+(\tilde{\mathbbm{n}}^\top+\tilde{k}\mathbbm{q}^\top)\Sigma(t,x,y,z+\Delta,u)\\
		&\quad +\tilde{k}g(t,x,y,z+\Delta,u)+\frac{1}{2}\big[\Sigma(t,x,y,z+\Delta,u)-\Sigma(t,\bar{X},\bar{Y},\bar{Z},\bar{u})\big]^\top\\
		&\quad\times(\tilde{\mathbbm{M}}+\tilde{k}\mathbbm{P})\big[\Sigma(t,x,y,z+\Delta,u)-\Sigma(t,\bar{X},\bar{Y},\bar{Z},\bar{u})\big]\\
		&=\theta\zeta(t)\Big\{(m^\top+kp^\top)b(t)+l(t)+(n^\top(t)+\theta\kappa m^\top+kq^\top)\sigma(t)+kg(t)\\
		&\qquad +\frac{1}{2}\big[\sigma(t,x,y,z+\Delta,u)-\sigma(t,\bar{X},\bar{Y},\bar{Z},\bar{u})\big]^\top(M+\theta mm^\top+kP)\\
		&\qquad \times\big[\sigma(t,x,y,z+\Delta,u)-\sigma(t,\bar{X},\bar{Y},\bar{Z},\bar{u})\big]\Big\}\\
		&:=\theta\zeta(t)\mathcal{H}_\theta^{RS}(t,x,y,z,u,p,q,k,m,n,P,M,\kappa).
	\end{aligned}
\end{equation}
By Lemma \ref{SMP-MayerType}, we arrive at
\begin{equation*}
	\begin{aligned}
		&\tilde{\mathcal{H}}^{RN}(t,\bar{X}(t),\bar{Y}(t),\bar{Z}(t),u,\mathbbm{p}(t),\mathbbm{q}(t),\mathbbm{P}(t),\tilde{k}(t),\tilde{\mathbbm{m}}(t),\tilde{\mathbbm{n}}(t),\tilde{\mathbbm{M}}(t))\\
		&-\tilde{\mathcal{H}}^{RN}(t,\bar{X}(t),\bar{Y}(t),\bar{Z}(t),\bar{u}(t),\mathbbm{p}(t),\mathbbm{q}(t),\mathbbm{P}(t),\tilde{k}(t),\tilde{\mathbbm{m}}(t),\tilde{\mathbbm{n}}(t),\tilde{\mathbbm{M}}(t))\geq0,\\
		&\hspace{6cm}\forall u\in U,\quad a.e.,\quad \mathbb{P}\mbox{-}a.s..
	\end{aligned}
\end{equation*}

Notice that $\zeta(t)>0$, $\mathbb{P}$-a.s., $t\in[0,T]$, then we get the maximum condition (\ref{maximum condition-RS}):
\begin{equation*}
	\begin{aligned}
		&\theta\,\mathcal{H}_\theta^{RS}(t,\bar{X}(t),\bar{Y}(t),\bar{Z}(t),u,p(t),q(t),m(t),n(t),k(t),P(t),M(t),\kappa(t))\\
		&\geq\theta\,\mathcal{H}_\theta^{RS}(t,\bar{X}(t),\bar{Y}(t),\bar{Z}(t),\bar{u}(t),p(t),q(t),m(t),n(t),k(t),P(t),M(t),\kappa(t)),\\
		&\hspace{6cm}\forall u\in U,\quad a.e.,\quad \mathbb{P}\mbox{-}a.s..
	\end{aligned}
\end{equation*}
This completes the proof of Theorem \ref{SMP-RS}.

\section{A linear-exponential-quadratic example}

In this section, we discuss an LEQ problem. By using Girsanov's Theorem and completion-of-squares technique, optimal feedback control and optimal cost are obtained. The Riccati equations and related ordinary differential equations (ODEs, for short) are established in risk-sensitive forms, which means the risk-sensitive parameter appears in their equations.

For convenience, we only consider the one-dimension case and the time variable may be omitted when no confusion appears. The state equation is the following fully coupled linear FBSDE:
\begin{equation}\label{LEQ-state-equation}
	\left\{
	\begin{aligned}
		dX(t)&=[A_1(t)X(t)+B_1(t)Y(t)+C_1(t)Z(t)+D_1(t)u(t)+b(t)]\,dt+\sigma(t)\,dW(t),\\
		-dY(t)&=[A_2(t)X(t)+B_2(t)Y(t)+C_2(t)Z(t)+D_2(t)u(t)+g(t)]\,dt\\
        &\quad -Z(t)\,dW(t),\quad t\in[0,T],\\
		X(0)&=x_0,\quad Y(T)=GX(T).
	\end{aligned}
	\right.
\end{equation}

Introduce the following cost functional, where all square terms and cross terms of $X$, $Y$, $Z$ and $u$ can appear in:
\begin{equation}\label{LEQ-csot functional}
	\begin{aligned}
		&J(u(\cdot))=\mathbb{E}\Bigg[\exp\bigg\{\frac{\theta}{2}\bigg(S_1X^2(T)+S_2Y^2(0)\\
        &\quad +\int_0^T\begin{pmatrix}X(t),Y(t),Z(t),u(t)\end{pmatrix}
        \begin{pmatrix}
			R_{11}(t)&R_{12}(t)  &R_{13}(t)  & R_{14}(t) \\
			R_{12}(t)	&R_{22}(t)  &R_{23}(t)  &R_{24}(t)  \\
			R_{13}(t)&R_{23}(t)  &R_{33}(t)  &R_{34}(t)  \\
			R_{14}(t)& R_{24}(t) &R_{34}(t)  & R_{44}(t)
		\end{pmatrix}\begin{pmatrix}
			X(t)\\
			Y(t)\\
			Z(t)\\
			u(t)
		\end{pmatrix}\,dt\bigg)\bigg\}\Bigg].
	\end{aligned}
\end{equation}
Accordingly, the following assumptions are imposed.
\begin{myassum}\label{LEQ-assum}
	\begin{itemize}
		\item[(i)] $\varphi_i(\cdot)$, $\varphi=A,B,C,D,i=1,2$, $\sigma(\cdot)$, $R_{ij}(\cdot)$, $i,j=1,2,3,4$ are deterministic and bounded, and $x_0$, $G$, $S_1$ and $S_2$ are constants;
		\item[(ii)] $R_{44}(\cdot)$ is uniformly positive definite; $\theta\textgreater0$ is fixed.
	\end{itemize}
\end{myassum}

Now we use the completion-of-squares technique to determine the optimal feedback control as well as the optimal cost. First, for an admissible quadruple  $(\bar{X}(\cdot),\bar{Y}(\cdot),\bar{Z}(\cdot),\bar{u}(\cdot))$, we decouple the state equation (\ref{LEQ-state-equation}) with
\begin{equation}\label{LEQ-decouple-state}
	\bar{Y}(\cdot)=\beta_1(\cdot)\bar{X}(\cdot)+\beta_2(\cdot),
\end{equation}
where the equations that $\beta_1(\cdot)$ and $\beta_2(\cdot)$ satisfy will be determined later such that $\beta_1(T)=G$, $\beta_2(T)=0$. By It\^o's formula, we can deduce that $\bar{Z}(\cdot)=\beta_1(\cdot)\sigma(\cdot)$, $\mathbb{P}\mbox{-}a.s..$ Hence, the SDE satisfied by $\bar{X}(\cdot)$ takes the following form:
\begin{equation}\label{LEQ-SDE-opti X}
	\left\{
	\begin{aligned}
		d\bar{X}(t)&=\big[(A_1(t)+B_1(t)\beta_1(t))\bar{X}(t)+B_1(t)\beta_2(t)+C_1(t)\beta_1(t)\sigma(t)+D_1(t)\bar{u}(t)+b(t)\big]\,dt\\
		&\quad +\sigma(t)\,dW(t),\quad t\in[0,T],\\
		X(0)&=x_0.
	\end{aligned}
	\right.
\end{equation}

As mentioned before, (\ref{Equation-gamma-kappa}) plays an important role in solving the risk-sensitive optimal control problem. In this LEQ problem of the present section, (\ref{Equation-gamma-kappa}) writes
\begin{equation}
	\left\{
	\begin{aligned}
		d\gamma(t)&=-\Bigg\{\frac{1}{2}\begin{pmatrix}
			\bar{X}\\
			\beta_1\bar{X}+\beta_2\\
			\beta_1\sigma\\
			\bar{u}
		\end{pmatrix}^{\top}\begin{pmatrix}
			R_{11}&R_{12}  &R_{13}  & R_{14} \\
			R_{12}	&R_{22}  &R_{23}  &R_{24}  \\
			R_{13}&R_{23} &R_{33}  &R_{34}  \\
			R_{14}& R_{24} &R_{34}  & R_{44}
		\end{pmatrix}\begin{pmatrix}
			\bar{X}\\
			\beta_1\bar{X}+\beta_2\\
			\beta_1\sigma\\
			\bar{u}
		\end{pmatrix}+\frac{\theta}{2}\kappa^2(t)\Bigg\}\,dt\\
		&\quad +\kappa(t)\,dW(t),\quad t\in[0,T],\\
		\gamma(T)&=\frac{1}{2}S_1\bar{X}^2(T)+\frac{1}{2}S_2\bar{Y}^2(0).
	\end{aligned}
	\right.
\end{equation}

Suppose we can express $\gamma(\cdot)$ as an explicit function of $\bar{X}(\cdot)$ as follows:
\begin{equation}\label{LEQ-gamma-opti X}
	\gamma(\cdot)=\frac{\alpha_1(\cdot)}{2}\bar{X}^2(\cdot)+\alpha_2(\cdot)\bar{X}(\cdot)+\alpha_3(\cdot),
\end{equation}
where $\alpha_1(\cdot)$, $\alpha_2(\cdot)$ and $\alpha_3(\cdot)$ are deterministic functions such that $\alpha_1(T)=S_1$, $\alpha_2(T)=0$ and $\alpha_3(T)=\frac{1}{2}S_2Y(0)^2=\frac{1}{2}S_2(\beta_1(0)x_0+\beta_2(0))^2$.

In view of (\ref{LEQ-SDE-opti X}), applying It\^o's formula to (\ref{LEQ-gamma-opti X}) and identifying the coefficients of diffusion terms, we conclude that
\begin{equation}\label{LEQ-kappa}
	\kappa(\cdot)=\sigma(\cdot)\big[\alpha_1(\cdot)X(\cdot)+\alpha_2(\cdot)\big],\quad \mathbb{P}\mbox{-}a.s..
\end{equation}
Substituting (\ref{LEQ-kappa}) into the drift term, and using the completion-of-squares technique, the following equality is obtained:
\begin{equation}
	\begin{aligned}
		0&=\Big[\frac{1}{2}\dot{\alpha}_1+\alpha_1(A_1+B_1\beta_1)+\frac{1}{2}(R_{11}+2R_{12}\beta_1+R_{22}\beta_1^2)-\frac{1}{2}R_{44}^{-1}(\alpha_1D_1+R_{14}+R_{24}\beta_1)^2\\
        &\quad\ +\frac{\theta}{2}\sigma^2\alpha_1^2\Big]X^2+\Big[\dot{\alpha}_2+\alpha_2(A_1+B_1\beta_1)+\alpha_1(B_1\beta_2+C_1\beta_1\sigma+b)\\
		&\quad\ +R_{12}\beta_2+R_{13}\beta_1\sigma+R_{22}\beta_1\beta_2+R_{23}\beta_1^2\sigma+\theta\sigma^2\alpha_1\alpha_2\\
		&\quad\ -R_{44}^{-1}(\alpha_1D_1+R_{14}+R_{24}\beta_1)(\alpha_2D_1+R_{24}\beta_2+R_{34}\beta_1\sigma)\Big]X\\
		&\ \ +\Big[\dot{\alpha}_3+\frac{1}{2}\alpha_1\sigma^2+\alpha_2(B_1\beta_2+C_1\beta_1\sigma+b)+\frac{1}{2}R_{22}\beta_2^2+R_{23}\sigma\beta_1\beta_2+\frac{1}{2}R_{33}\sigma^2\beta_1^2\\
		&\quad\ +\frac{\theta}{2}\sigma^2\alpha_2^2-\frac{1}{2}R_{44}^{-1}(\alpha_2D_1+R_{24}\beta_2+R_{34}\beta_1\sigma)\Big]\\
		&\ \ +\frac{1}{2}R_{44}\Big\{u+R_{44}^{-1}\big[(\alpha_1D_1+R_{14}+R_{24}\beta_1)X+(\alpha_2D_1+R_{24}\beta_2+R_{34}\beta_1\sigma)\big]\Big\}^2.
	\end{aligned}
\end{equation}
From above, by simply calculating, we can deduce the feedback representation of $\bar{u}(\cdot)$:
\begin{equation}\label{LEQ-feedback-control}
	\bar{u}=-R_{44}^{-1}\big[(\alpha_1D_1+R_{14}+R_{24}\beta_1)\bar{X}+(\alpha_2D_1+R_{24}\beta_2+R_{34}\beta_1\sigma)\big],\quad a.e.,\quad \mathbb{P}\mbox{-}a.s.,
\end{equation}
as well as a risk-sensitive Riccati equation along with two risk-sensitive ODEs:
\begin{equation}\label{LEQ-alpha1}
	\left\{
	\begin{aligned}
		&\dot{\alpha}_1(t)+2\mathcal{C}_1(t)\alpha_1(t)+2\mathcal{C}_2(t)\beta_1(t)+2\mathcal{C}_3(t)\alpha_1(t)\beta_1(t)+\mathcal{C}_4(t)\alpha_1^2(t)\\
		&\quad +\mathcal{C}_5(t)\beta_1^2(t)+\mathcal{C}_6(t)=0,\quad t\in[0,T],\\
		&\alpha_1(T)=S_1,
	\end{aligned}
	\right.
\end{equation}
\begin{equation}\label{LEQ-alpha2}
\hspace{-1.6cm}	\left\{
	\begin{aligned}
		&\dot{\alpha}_2(t)+\mathcal{C}_7(t)\alpha_2(t)+\mathcal{C}_8(t)\beta_2(t)+\mathcal{C}_9(t)=0,\quad t\in[0,T],\\
		&\alpha_2(T)=0,
	\end{aligned}
	\right.
\end{equation}
\begin{equation}\label{LEQ-alpha3}
\hspace{-5.7cm}	\left\{
	\begin{aligned}
		&\dot{\alpha}_3(t)+\mathcal{C}_{10}(t)=0,\quad t\in[0,T],\\
		&\alpha_3(T)=\frac{1}{2}S_2(\beta_1(0)x_0+\beta_2(0))^2,
	\end{aligned}
	\right.
\end{equation}
where
\begin{equation*}
	\begin{aligned}
		\mathcal{C}_1&:=A_1-R^{-1}_{44}D_1R_{14},\quad \mathcal{C}_2:=R_{12}-R^{-1}_{44}R_{14}R_{24},\quad \mathcal{C}_3:=B_1-R^{-1}_{44}D_1R_{24},\\
        \mathcal{C}_4&:=\theta\sigma^2-R^{-1}_{44}D_1^2,\quad \mathcal{C}_5:=R_{22}-R^{-1}_{44}R^2_{24},\quad \mathcal{C}_6:=R_{11}-R^{-1}_{44}R^2_{14},\\
		\mathcal{C}_7&:=A_1+B_1\beta_1+\theta\sigma^2\alpha_1-R^{-1}_{44}D_1(D_1\alpha_1+R_{14}+R_{24}\beta_1),\\
		\mathcal{C}_8&:=B_1\alpha_1+R_{12}+R_{22}\beta_1-R^{-1}_{44}R_{24}(D_1\alpha_1+R_{14}+R_{24}\beta_1),\\
		\mathcal{C}_9&:=C_1\sigma\alpha_1\beta_1+\alpha_1b+R_{13}\sigma\beta_1+R_{23}\sigma\beta_1^2-R^{-1}_{44}R_{34}\sigma(D_1\alpha_1+R_{14}+R_{24}\beta_1)\beta_1,\\
		\mathcal{C}_{10}&:=\frac{1}{2}\alpha_1\sigma^2+\alpha_2(B_1\beta_2+C_1\beta_1\sigma+b)+\frac{1}{2}R_{22}\beta_2^2+R_{23}\sigma\beta_1\beta_2+\frac{1}{2}R_{33}\sigma^2\beta_1^2\\
		&\quad +\frac{\theta}{2}\sigma^2\alpha_2^2-\frac{1}{2}R_{44}^{-1}(\alpha_2D_1+R_{24}\beta_2+R_{34}\beta_1\sigma)^2.
	\end{aligned}
\end{equation*}

Taking the state feedback representation (\ref{LEQ-feedback-control}) and the decoupling relation (\ref{LEQ-decouple-state}) into (\ref{LEQ-state-equation}), the optimal dynamics solves the linear FBSDE:
\begin{equation}
	\left\{
	\begin{aligned}
		d\bar{X}(t)&=[\mathcal{A}_1(t)\bar{X}(t)+\mathcal{B}_1(t)]\,dt+\sigma(t)\,dW(t),\\
		-d\bar{Y}(t)&=[\mathcal{A}_2(t)\bar{X}(t)+\mathcal{B}_2(t)]\,dt-\bar{Z}(t)\,dW(t),\quad t\in[0,T],\\
		X(0)&=x_0,\quad Y(T)=G\bar{X}(T),
	\end{aligned}
	\right.
\end{equation}
where
\begin{equation*}
	\begin{aligned}
		\mathcal{A}_1&:=A_1+B_1\beta_1-R_{44}^{-1}D_1(\alpha_1D_1+R_{14}+R_{24}\beta_1),\\
		\mathcal{B}_1&:=B_1\beta_2+C_1\beta_1\sigma+b-R_{44}^{-1}D_1(\alpha_2D_1+R_{24}\beta_2+R_{34}\beta_1\sigma),\\
		\mathcal{A}_2&:=A_2+B_2\beta_1-R_{44}^{-1}D_2(\alpha_1D_1+R_{14}+R_{24}\beta_1),\\
		\mathcal{B}_2&:=B_2\beta_2+C_2\beta_1\sigma+g-R_{44}^{-1}D_2(\alpha_2D_1+R_{24}\beta_2+R_{34}\beta_1\sigma).
	\end{aligned}
\end{equation*}
Applying It\^o's formula to (\ref{LEQ-decouple-state}) and identifying the coefficients of drift terms, we have
%\begin{equation}
%	\begin{aligned}
%		&d\big[\beta_1(t)\bar{X}(t)+\beta_2(t)\big]\\
%        &=\big[(\dot{\beta}_1(t)+\beta_1(t)\mathcal{A}_1(t))\bar{X}(t)+\beta_1(t)\mathcal{B}_1(t)+\dot{\beta}_2(t)\big]\,dt+\sigma(t)\beta_1(t)\,dW(t)\\
%		&=d\bar{Y}(t)=[-\mathcal{A}_2(t)\bar{X}(t)-\mathcal{B}_2(t)]\,dt+\bar{Z}(t)\,dW(t),
%	\end{aligned}
%\end{equation}
%which implies that
\begin{equation}
	\big[\dot{\beta}_1(t)+\beta_1(t)\mathcal{A}_1(t)+\mathcal{A}_2(t)\big]\bar{X}(t)+\dot{\beta}_2(t)+\beta_1(t)\mathcal{B}_1(t)+\mathcal{B}_2(t)=0,\quad a.e..
\end{equation}
Let the coefficients for $\bar{X}(\cdot)$ and the constant term be 0, then the decoupling Riccati equation and ODE that $\beta_1(\cdot)$ and $\beta_2(\cdot)$ satisfy take the following form:
\begin{equation}\label{LEQ-beta1}
\hspace{-1.2cm}	\left\{
	\begin{aligned}
		&\dot{\beta}_1(t)+\mathcal{C}_{11}(t)\beta_1(t)-\mathcal{C}_{12}(t)\alpha_1(t)-\mathcal{C}_{13}(t)\alpha_1(t)\beta_1(t)\\
		&\quad +\mathcal{C}_{14}(t)\beta_1^2(t)+\mathcal{C}_{15}(t)=0,\quad t\in[0,T],\\
		&\beta_1(T)=G,
	\end{aligned}
	\right.
\end{equation}
and
\begin{equation}\label{LEQ-beta2}
	\left\{
	\begin{aligned}
		&\dot{\beta}_2(t)+\mathcal{C}_{16}(t)\beta_2(t)-\mathcal{C}_{17}(t)\alpha_2(t)+\mathcal{C}_{18}(t)=0,\quad t\in[0,T],\\
		&\beta_2(T)=0,
	\end{aligned}
	\right.
\end{equation}
respectively, where
\begin{equation*}
	\begin{aligned}
		\mathcal{C}_{11}&=A_1+B_2-R^{-1}_{44}(D_1R_{14}+D_2R_{24}),\quad \mathcal{C}_{12}=R^{-1}_{44}D_1D_2,\\
        \mathcal{C}_{13}&=R^{-1}_{44}D_1^2,\quad \mathcal{C}_{14}=B_1-R^{-1}_{44}D_1R_{24},\quad \mathcal{C}_{15}=A_2-R^{-1}_{44}D_2R_{14},\\
        \mathcal{C}_{16}&=B_1\beta_1+B_2-R^{-1}_{44}R_{24}(D_1\beta_1+D_2),\quad \mathcal{C}_{17}=R^{-1}_{44}D_1(D_1\beta_1+D_2),\\ \mathcal{C}_{18}&=(C_1\beta_1+C_2)\sigma\beta_1+b\beta_1+g-R^{-1}_{44}R_{34}\sigma\beta_1(D_1\beta_1+D_2).
	\end{aligned}
\end{equation*}

\begin{Remark}\label{Remark-Riccati-solv}
	Let us give some remarks on the coupled feature of the above equations. In fact, if the coupled Riccati equations (\ref{LEQ-alpha1}), (\ref{LEQ-beta1}) admit unique solutions $\alpha_1(\cdot)$, $\beta_1(\cdot)$ respectively, then $\mathcal{C}_7(\cdot)$, $\mathcal{C}_8(\cdot)$, $\mathcal{C}_9(\cdot)$, $\mathcal{C}_{16}(\cdot)$, $\mathcal{C}_{17}(\cdot)$ and $\mathcal{C}_{18}(\cdot)$ become determined functions. As a result, (\ref{LEQ-alpha2}) and (\ref{LEQ-beta2}) become coupled first-order linear ODEs, whose solvability are easy to get. Finally, (\ref{LEQ-alpha3}) is also solvable since $\mathcal{C}_{10}(\cdot)$ is known. A local solvability result of (\ref{LEQ-alpha1}) and (\ref{LEQ-beta1}) will be given later by Picard-Lindel$\ddot{o}$f's Theorem.
\end{Remark}

Now we present the main result of this section.

\begin{mythm}\label{LEQ-Thm}
	Suppose Assumption \ref{LEQ-assum} holds. Let (\ref{LEQ-alpha1})-(\ref{LEQ-alpha3}), (\ref{LEQ-beta1}) and (\ref{LEQ-beta2}) have unique solutions $\alpha_1(\cdot)$, $\alpha_2(\cdot)$, $\alpha_3(\cdot)$, $\beta_1(\cdot)$ and $\beta_2(\cdot)$ respectively. Then the LEQ problem given by (\ref{LEQ-state-equation}) and (\ref{LEQ-csot functional}) has a unique optimal feedback control:
\begin{equation}\label{unique optimal feedback control}
		\begin{aligned}
			\bar{u}(t)=&-R_{44}^{-1}(t)\Big\{\big[\alpha_1(t)D_1(t)+R_{14}(t)+R_{24}(t)\beta_1(t)\big]\bar{X}(t)\\
			&\qquad +\alpha_2(t)D_1(t)+R_{24}(t)\beta_2(t)+R_{34}(t)\beta_1(t)\sigma(t)\Big\},\quad a.e.,\quad \mathbb{P}\mbox{-}a.s..
		\end{aligned}
\end{equation}
Furthermore, the optimal cost is
\begin{equation}
	J(\bar{u}(\cdot))=\exp\bigg\{\theta\bigg(\frac{1}{2}\alpha_1(0)x_0^2+\alpha_2(0)x_0+\alpha_3(0)\bigg)\bigg\}.
\end{equation}
\end{mythm}

\begin{proof}
For any given admissible control $u(\cdot)$, denoting $(X^u(\cdot),Y^u(\cdot),Z^u(\cdot))$ as the corresponding solution of (\ref{LEQ-state-equation}). Applying It\^o's formula to $\frac{\alpha_1(\cdot)}{2}(X^u(\cdot))^2+\alpha_2(\cdot)X^u(\cdot)+\alpha_3(\cdot)$ and integrating it from $0$ to $T$, it follows that
\begin{equation}\label{LEQ-initial-terminal-cost-Ito}
		\frac{S_1}{2}(X^u(T))^2+\frac{S_2}{2}(Y^u(0))^2=\frac{1}{\theta}\ln\Upsilon(0)+\int_0^T\,d\bigg(\frac{\alpha_1(t)}{2}(X^u(t))^2+\alpha_2(t)X^u(t)+\alpha_3(t)\bigg),
\end{equation}
where $\Upsilon(0):=\exp\big\{\theta\big(\frac{1}{2}\alpha_1(0)x_0^2+\alpha_2(0)x_0+\alpha_3(0)\big)\big\}$ is a constant.

In view of (\ref{LEQ-SDE-opti X}) and (\ref{LEQ-alpha1})-(\ref{LEQ-alpha3}), substituting (\ref{LEQ-initial-terminal-cost-Ito}) into (\ref{LEQ-csot functional}) and using the completion-of-squares technique, the corresponding cost becomes:
\begin{equation*}
	\begin{aligned}
		J(u(\cdot))&=\Upsilon(0)\mathbb{E}\bigg[\exp\bigg\{\theta\int_0^T\bigg[-\frac{\theta}{2}\sigma^2\alpha_1^2(X^u)^2-\theta\sigma^2\alpha_1\alpha_2X^u-\frac{\theta}{2}\sigma^2\alpha_2^2\\
		&\qquad +\frac{1}{2}R_{44}\big\{u+R_{44}^{-1}\big[(\alpha_1D_1+R_{14}+R_{24}\beta_1)X^u+\alpha_2D_1+R_{24}\beta_2+R_{34}\beta_1\sigma\big]\big\}^2\bigg]\,dt\\
		&\qquad +\int_0^T\theta\sigma(\alpha_1X^u+\alpha_2)\,dW(t)\bigg\}\bigg]\\
		&=\Upsilon(0)\mathbb{E}\bigg[\exp\bigg\{\theta\int_0^T\bigg[\frac{1}{2}R_{44}\big\{u+R_{44}^{-1}\big[(\alpha_1D_1+R_{14}+R_{24}\beta_1)X^u\\
		&\qquad +\alpha_2D_1+R_{24}\beta_2+R_{34}\beta_1\sigma\big]\big\}^2\bigg]\,dt\\
		&\qquad +\int_0^T\theta\sigma(\alpha_1X^u+\alpha_2)\,dW(t)-\frac{1}{2}\int_0^T\theta^2\sigma^2(\alpha_1X^u+\alpha_2)^2 \,dt\bigg\}\bigg]\\
		&=\Upsilon(0)\mathbb{E}^u\bigg[\exp\bigg\{\theta\int_0^T\bigg[\frac{1}{2}R_{44}\big\{u+R_{44}^{-1}\big[(\alpha_1D_1+R_{14}+R_{24}\beta_1)X^u\\
		&\qquad +\alpha_2D_1+R_{24}\beta_2+R_{34}\beta_1\sigma\big]\big\}^2\bigg]\,dt\bigg\}\bigg],
	\end{aligned}
\end{equation*}
where $\mathbb{E}^u$ is the expectation with respect to $\mathbb{P}^u$ given by
\begin{equation*}
	\frac{d\mathbb{P}^u}{d\mathbb{P}}:=\exp\bigg\{\int_0^T\theta\sigma(t)\big(\alpha_1(t)X^u(t)+\alpha_2(t)\big)\,dW(t)-\frac{1}{2}\int_0^T\theta^2\sigma^2(t)\big(\alpha_1(t)X^u(t)+\alpha_2(t)\big)^2\,dt\bigg\}.
\end{equation*}

From the above calculation, it is obvious that for $\bar{u}(\cdot)$ given by (\ref{unique optimal feedback control}), the related optimal cost is $J(\bar{u}(\cdot))=\Upsilon(0)$.
Hence, for any given admissible $u(\cdot)$, we can conclude that
\begin{equation*}
	J(u(\cdot))\geq J(\bar{u}(\cdot)).
\end{equation*}
The proof is complete.
\end{proof}

As evident from the previous discussions, the optimal control and optimal cost depend heavily on solutions of (\ref{LEQ-alpha1})-(\ref{LEQ-alpha3}) and (\ref{LEQ-beta1})-(\ref{LEQ-beta2}). From Remark \ref{Remark-Riccati-solv}, we only need to discuss the unique solvability of coupled Riccati equations (\ref{LEQ-alpha1}) and (\ref{LEQ-beta1}). Noting that (\ref{LEQ-alpha1}) and (\ref{LEQ-beta1}) are not in classical form, we do not have results about their solvability in general now. However, next we investigate the local solvability of them. Set
\begin{equation*}
		K:=\begin{bmatrix}
			\alpha_1&0  \\
			0&\beta_1
		\end{bmatrix},\
		J:=\begin{bmatrix}
			0&1  \\
			1&0
		\end{bmatrix},\
		K_T:=\begin{bmatrix}
			S_1&0  \\
			0&G
		\end{bmatrix},\
		\mathcal{D}_1:=\begin{bmatrix}
			\mathcal{C}_1&0  \\
			0&\frac{1}{2}\mathcal{C}_{11}
		\end{bmatrix},\
		\mathcal{D}_2:=\begin{bmatrix}
			2\mathcal{C}_2&0  \\
			0&-\mathcal{C}_{12}
		\end{bmatrix},
\end{equation*}
\begin{equation*}
		\mathcal{D}_3:=\begin{bmatrix}
			2\mathcal{C}_3&0  \\
			0&-\mathcal{C}_{13}
		\end{bmatrix},\
		\mathcal{D}_4:=\begin{bmatrix}
			\mathcal{C}_4&0  \\
			0&\frac{1}{2}\mathcal{C}_{14}
		\end{bmatrix},\
		\mathcal{D}_5:=\begin{bmatrix}
			\mathcal{C}_5&0  \\
			0&0
		\end{bmatrix},\
		\mathcal{D}_6:=\begin{bmatrix}
			\mathcal{C}_6&0  \\
			0&\mathcal{C}_{15}
		\end{bmatrix}.
\end{equation*}
Then equations (\ref{LEQ-alpha1}) and (\ref{LEQ-beta1}) are equal to the following
\begin{equation}\label{LEQ_Riccati-HighD}
		\left\{
		\begin{aligned}
			\dot{K}&(t)+\mathcal{D}_1(t)K(t)+K(t)\mathcal{D}_1(t)+\mathcal{D}_2(t)JK(t)J+K(t)J\mathcal{D}_3(t)K(t)J\\
			&+K\mathcal{D}_4(t)K(t)+JK(t)J\mathcal{D}_5(t)JK(t)J+\mathcal{D}_6(t)=0,\quad t\in[0,T],\\
			K&(T)=K_T.
		\end{aligned}
		\right.
\end{equation}

Since $\mathcal{G}(t,K):=\mathcal{D}_1K+K\mathcal{D}_1+\mathcal{D}_2JKJ+KJ\mathcal{D}_3KJ+K\mathcal{D}_4K+JKJ\mathcal{D}_5JKJ+\mathcal{D}_6$
is continuous in $t$ and uniformly locally Lipschitz continuous in $K$. According to Picard-Lindel$\ddot{o}$f's Theorem (See, for example, Theorem 15.2 in Agarwal and O'Regan \cite{AO08}), we know that there
exists an interval $[T-\eta,T]\subseteq[0,T]$ such that the terminal value problem (\ref{LEQ_Riccati-HighD}) as well as (\ref{LEQ-alpha1}) and (\ref{LEQ-beta1}) has a unique solution on $[T-\eta,T]$.

\section{Concluding remarks}

In this paper, we have obtained a global maximum principle for a kind of risk-sensitive optimal control problem for fully coupled forward-backward stochastic systems. The control variable can enter the diffusion term of the state equation and the control domain is not necessarily convex. A new global maximum principle is obtained without assuming that the value function is smooth. The maximum condition, the first- and second-order adjoint equations heavily depend on the risk-sensitive index. An optimal control problem with a fully coupled linear forward-backward stochastic system and an exponential-quadratic cost functional is discussed. The optimal feedback control and optimal cost are obtained by using Girsanov's Theorem and completion-of-squares approach via risk-sensitive Riccati equations. A local solvability result of coupled risk-sensitive Riccati equations (\ref{LEQ-alpha1}) and (\ref{LEQ-beta1}) is given by Picard-Lindel$\ddot{o}$f's Theorem.

Possible future research includes the partially observed problems, mean field type control and differential games. We will consider these topics in the near future.

\end{document}